\documentclass[a4paper,12pt]{article}

\usepackage{amsthm}
\usepackage{amsmath}
\usepackage{amssymb}
\usepackage{mathrsfs}
\usepackage{graphicx}
\usepackage{hyperref}
\usepackage{bm}
\usepackage{dsfont}
\usepackage{enumitem}

\usepackage[english]{babel}
\usepackage[T1]{fontenc}
\usepackage[utf8]{inputenc}

\usepackage[all]{xy}
\usepackage{caption}
\usepackage{graphicx}
\usepackage{tikz}
\usepackage{tikz-cd}
\usetikzlibrary{patterns}
\usetikzlibrary{arrows}
\usetikzlibrary{arrows.meta}
\usepackage{comment}

\usepackage{geometry}
\usepackage[square,numbers]{natbib}
\newcommand{\pdim}{\mathrm{pdim}}
\newcommand{\ext}{\mathrm{Ext}}
\newcommand{\tor}{\mathrm{Tor}}

\newcommand{\HH}{HH}
\newcommand{\Homo}{H}
\newcommand{\Hom}{\mathrm{Hom}}

\newcommand{\cd}{\mathrm{cd}}

\newcommand{\Vect}{\mathrm{Vect}}
\newcommand{\Ker}{\mathrm{Ker}}
\newcommand{\id}{\mathrm{id}}

\newtheorem{thm}{Theorem}[section]

\newtheorem{coro}[thm]{Corollary}

\newtheorem{lemma}[thm]{Lemma}

\newtheorem{prop}[thm]{Proposition}

\theoremstyle{remark}
\newtheorem{rem}[thm]{Remark}

\newtheorem{ex}[thm]{Example}

\theoremstyle{definition}
\newtheorem{defi}[thm]{Definition}

\newtheorem*{nota}{Notations}

\title{Homological duality and exact sequences of Hopf algebras}
\date{}
\author{Julian Le Clainche\thanks{Université Clermont Auvergne, CNRS, LMBP, F-63000 CLERMONT-FERRAND, FRANCE}}

\begin{document}

\maketitle

\begin{abstract}
We study the stability of homological duality properties of Hopf algebras under extensions.
\end{abstract}


\section{Introduction}\label{Introduction}

Homological duality is a general phenomenon producing a duality between the cohomology and homology spaces of an algebraic system, analogous to Poincaré duality for closed manifolds in algebraic topology. A first axiomatization is due to Bieri-Eckmann \cite{bieri_groups_1973} in group cohomology, while  homological duality in Hochschild cohomology for algebras arose in Van den Bergh's work \cite{van_den_bergh_relation_1998,van_den_bergh_erratum_2002}. Since then, many developments have followed, including the important case of (twisted) Calabi-Yau algebras \cite{ginzburg_calabi-yau_nodate}.

Homological duality for groups is known to be well-behaved under extensions \cite{bieri_groups_1973,bie76}. Our main result, which produces new examples of (Hopf) algebras having homological duality, generalizes these results of Bieri-Eckmann to the setting of exact sequences of Hopf algebras, as follows, where $\cd$ denotes the Hochschild cohomological dimension of an algebra. 




\begin{thm}\label{th seq}
    Let $k \longrightarrow B \overset{i}{\longrightarrow} A \overset{p}{\longrightarrow} H \longrightarrow k$ be an exact sequence of Hopf algebras with bijective antipodes. Assume that $A$ is faithfully flat as a left and right $B$-module, and that $H$ and $B$ are smooth algebras. Then the following assertions are equivalent:
    \begin{enumerate}[label = (\roman*)]
        \item $B$ and $H$ have homological duality,
        \item $A$ has homological duality.
    \end{enumerate}    
If these conditions hold, we have
$\cd (A) = \cd(B) + \cd(H).$
\end{thm}


Apart from a smoothness criterion due to Bieri-Eckmann \cite{Bieri_Eckmann74}, the main tools we use to prove Theorem \ref{th seq} are natural Hopf algebraic generalizations of the Lyndon-Hochschild-Serre spectral sequences. These spectral sequences are not new, because they are special cases of those of Stefan \cite{stefan_hochschild_1995} in the more general setting of Hopf-Galois extensions, but we work them in detail, on the one hand for the sake of completeness, and on the other hand because there are some specific technicalities  (notably regarding certain $H$-actions) that are not directly viewable, or that we have not been able to view, in \cite{stefan_hochschild_1995}, and are  needed in the proof of Theorem \ref{th seq}.

During the last phase of writing this paper, we were informed of the recent work of R. Zhu \cite{zhu}, where a result similar to the implication (i) $\Rightarrow$ (ii) in Theorem \ref{th seq} is proved (Theorem 2.15 in \cite{zhu}), in the more general framework of Hopf-Galois extensions. We still think that there is some interest in our direct approach in the case of exact sequences of Hopf algebras, on one hand because many considerations can be made more explicit than in the general Hopf-Galois case, and on the other hand because we deal with a more general version of homological duality (as suggested by the duality groups of Bieri-Eckmann) than the twisted Calabi-Yau or Van den Bergh duality considered in \cite{zhu}.

The paper is organized as follows. Section 2 consists of reminders and preliminaries about Hochschild (co)homology and homological duality. In Section 3 we recall the definition of exact sequences of Hopf algebras and we introduce the generalized Lyndon-Hochschild-Serre spectral sequences in Hopf algebra (co)homology. In Section 4 we establish a smoothness result for a Hopf algebras involved in an exact sequence. Section 5 is devoted to the proof of Theorem \ref{th seq}. In the final Section 6 we present some examples and applications of Theorem \ref{th seq}.

\begin{nota}
Throughout the paper, we work over a field $k$, and all algebras are (unital) $k$-algebras.
If $A$ is an algebra, the opposite algebra is denoted $A^{op}$ and the enveloping algebra $A\otimes A^{op}$ is denoted $A^e$.
The category of right (resp. left) $A$-modules is denoted $\mathcal{M}_{A}$ (resp. ${}_A \mathcal{M}$).
We have $\mathcal M_A={_{A^{op}}\mathcal{M}}$, hence any definition we give for left modules has an obvious analogue for right modules, that we will not necessarily give.
An $A$-bimodule structure is equivalent to a left (or right) $A^e$-module structure. Indeed, if $M$ is an $A$-bimodule, a left (resp. right) $A^e$-module structure on $M$ is defined by $$(a\otimes b) \cdot m = amb \text{ (resp. } m\cdot (a\otimes b) = bma\text{ )} \quad \mathrm{for} \quad a,b\in A \quad \mathrm{and}\quad m\in M.$$
The  category of $A$-bimodule is thus identified with the category $\mathcal{M}_{A^e}$, or with the category ${}_{A^e}\mathcal{M}$.

If $M$ and $N$ are left $A$-modules, $\ext$ spaces (\cite{weibel_introduction_2008}) are denoted $\ext_A^\bullet (M,N)$ and if $P$ is a right $A$-module $\tor$ spaces are denoted $\tor^A_\bullet(P,M)$. 
We will also need to consider $\ext$ spaces in categories of right modules, which we denote $\ext^\bullet_{A^{op}}(-,-)$.

 

If $H$ is a Hopf algebra, its comultiplication, counit and antipode are  denoted $\Delta, ~ \varepsilon$ and $S$ and we will use Sweedler notation in the usual way, i.e. for $h \in H$, we write $\Delta(h)=h_{(1)}\otimes h_{(2)}$. See \cite{montgomery_hopf_1993}.





 

\end{nota}

\section{Hochschild cohomology and homological duality}\label{Hochschild}

In this section, we briefly recall some notations and basic definitions about Hochschild (co)homology and homological duality for algebras.

\subsection{Finiteness conditions on modules}


\begin{defi}
    Let $A$ be an algebra. The \textit{projective dimension} of an $A$-module $M$ is defined by 
        \[\pdim_A(M) := \min\left\{ n \in \mathbb{N},~ M \text{ admits a length} \ n \ \text{resolution by projective } A\text{-modules} \right\} \in \mathbb{N}\cup \{\infty\}.\]
\end{defi}

The projective dimension can as well be characterized by
\begin{align*}
\pdim_A(M) &= \min\left\{ n \in \mathbb{N},~ \ext_A^{n+1}(M,N)=\{0\} \text{ for any $A$-module $N$}  \right\}  \\
& = \max \left\{ n \in \mathbb{N},~ \ext^{n}_A(M,N)\not=\{0\} \text{ for some $A$-module $N$}  \right\} 
\end{align*}
and if $\pdim_A(M)$ if finite, we have as well
\[ \pdim_A(M) = \max \left\{ n \in \mathbb{N},~ \ext^{n}_A(M,F)\not=\{0\} \text{ for some free $A$-module $F$}  \right\}\]

We now recall various finiteness conditions on modules.

\begin{defi}
    Let $A$ be an algebra and let $M$ be an $A$-module. 
    \begin{enumerate}
        \item The $A$-module $M$ is said to be \textit{of type $FP_\infty$} if it admits a projective resolution $P_\bullet \to M$ with $P_i$ finitely generated for all $i$. 
        \item The $A$-module $M$ is said to be \textit{of type FP} if it admits a finitely generated projective resolution of finite length.
    \end{enumerate}
\end{defi}

The following result characterizes modules of type $FP$ among those of type $FP_\infty$, see e.g \cite[Chapter VIII]{brown_cohomology_1982}.

\begin{prop}\label{prop:FP-FPinfty}
Let $A$ be an algebra. An $A$-module $M$ is of type $FP$ if and only if it is of type $FP_\infty$ and $\pdim_A(M)$ is finite. In this case we have     
\[\pdim_A(M) = \max\left\{ n \in \mathbb{N},~ \ext_A^{n}(M,A)\not=\{0\}   \right\}\]
\end{prop}

The following result is \cite[Corollary 1.6]{bieri_homological_1983}. The implication (iii)$\Rightarrow$(i), due to Bieri-Eckmann \cite{Bieri_Eckmann74}, gives an effective condition to check that a module is of type $FP_\infty$ and  will be useful in section \ref{Sec:Smoothness}.

\begin{prop}\label{characterization FP}
Let $A$ be an algebra and $M$ a left (resp. right) $A$-module. The following conditions are equivalent: 
\begin{enumerate}[label= (\roman*)]
    \item The left (resp. right) $A$-module $M$ is of type $FP_\infty$;
    \item For any direct product $\prod_{i\in I}N_i$ of right (resp. left) $A$-modules, the natural map  $\tor_k^A(\prod_{i \in I}N_i,M)\to \prod_{i\in I}\tor_k^A(N_i,M)$ (resp. $\tor_k^A(M,\prod_{i \in I}N_i)\to \prod_{i\in I}\tor_k^A(M,N_i)$ is an isomorphism for all $k\geq0$;
    \item For any direct product $\prod A$ of arbitrary many copies of $A$, we have $\tor_i^{A}(\prod A, M)= 0$ (resp. $\tor_i^{A}(M, \prod A)= 0$) for $i\geq 1$ and the natural map $\left(\prod A \right) \otimes_{A} M \to \prod M$(resp. $M \otimes_{A} \left(\prod A\right) \to \prod M$) is an isomorphism.
\end{enumerate}
\end{prop}

We have similar results for $\ext$ functors, but we will only use the following particular statement:

\begin{prop}\label{FP ext}
    Let $A$ be an algebra and $M$ a right (resp. left) $A$-module. If $M$ is of type $FP_\infty$, then $\ext_{A^{op}}(M,-)$ (resp. $\ext_A(M,-)$) commutes with direct sums.
\end{prop}

\subsection{Poincaré duality modules}

 In this subsection we recall the notion of Poincaré duality module, from \cite[Section V]{dicks_groups_1989}. This is the basic notion involved in homological duality questions. We fix an algebra $A$.

\begin{defi}\label{def:hdm}
A left $A$-module $X $ is said to be a \textsl{Poincaré duality module of dimension $d\geq 0$} if the following conditions hold:
\begin{enumerate}[label= (\roman*)]
	\item $X$ is of type $FP$ and $\pdim_A(X) = d$;
	\item There exists a right $A$-module $D_X$ together with, for any left $A$-module $N$ and any $i\in \{0, \ldots , d\}$, functorial isomorphims 
	$$\ext^i_A(X, N)\simeq \tor^A_{d-i}(D_X,N).$$
	\end{enumerate}
A right $A$-module $D_X$ as above is called a \textsl{dualizing module for $X$}.
\end{defi}


\begin{rem}\label{rem:unidualizing}
Let $X$ be a Poincaré duality module of dimension $d\geq 0$. If $X$ is an $A$-$S$-bimodule, with $S$ another algebra, the functoriality condition  in Definition \ref{def:hdm} implies that the isomorphism
$$\ext^i_A(X, N)\simeq \tor^A_{d-i}(D_X,N)$$
is right $S$-linear.
In particular, we have
$$\ext^i_A(X, A)\simeq \tor^A_{d-i}(D_X,A)\simeq \begin{cases} \{0\} \ \textrm{if} \ i\not=d \\
D_X \ \textrm{if} \  i=d
\end{cases}$$
where the  isomorphism is right $A$-linear. Therefore a dualizing module as in Definition \ref{def:hdm}  is unique up to isomorphism. 
\end{rem}

The following result is \cite[Theorem 2.21]{dicks_groups_1989}. For the sake of completeness, we give a direct proof of $(ii) \Rightarrow (i)$ following classical arguments of the setting of duality groups \cite{brown_cohomology_1982}.

\begin{thm}\label{thm:hdm}
Let $X$ be a left $A$-module of type $FP$ with $\pdim(X) = d\geq0$. The following assertions are equivalent:
\begin{enumerate}[label= (\roman*)]
\item $X$ is a Poincaré duality module of dimension $d$.
\item We have $\ext_A^i(X,A)=\{0\}$ if $i\not=d$.
\end{enumerate}	
In this case, we have $D_X\simeq \ext_A^d(X,A)$ as right $A$-modules.
\end{thm}

\begin{proof}
$(i) \Rightarrow (ii)$ and the last assertion follow from the remark above.
Assume that $(ii)$ holds,
and consider a resolution of $X$
$$0 \rightarrow P_d \rightarrow P_{d-1} \rightarrow \cdots \rightarrow P_1 \rightarrow P_0 \rightarrow X \rightarrow 0$$
by finitely generated projective left $A$-modules. Consider the (projective) right $A$-modules 
$$\overline{P_i} = {\rm Hom}_A(P_i,A)$$
and the corresponding complex 
$$0 \rightarrow \overline{P_0} \rightarrow \overline{P_1} \rightarrow \cdots \rightarrow \overline{P_{d-1}} \rightarrow \overline{P_d} \rightarrow 0$$
We have $H^*(\overline{P})= {\ext}^*_A(X ,A)$, and (ii) ensures that
$$0 \rightarrow \overline{P_0} \rightarrow \overline{P_1} \rightarrow \cdots \rightarrow \overline{P_{d-1}} \rightarrow \overline{P_d} \rightarrow {\ext}^d_A(X, A)\rightarrow 0$$
is a resolution of ${\ext}^d_A(X, A)$ by projective right $A$-modules.
We then have, for a left $A$-module $N$ 
\begin{align*}
{\tor}_{d-i}^A({\ext}^d_A(X, A),N) &\simeq H^i(\overline{P} \otimes_A N) = H^i({\rm Hom}_A(P,A)\otimes_A N)\\
& \simeq H^i({\rm Hom}_A(P,N)) \ (\textrm{finite generation and projectivity} \ {\rm of} \ P_i) \\
& \simeq {\ext}_A^i(X,N).
\end{align*}
This proves (i), with $D_X= {\ext}^d_A(X, A)$.
	\end{proof}

\subsection{Hochschild (co)homology and homological duality}\label{Hochschild def}

The Hochschild (co)homology spaces (\cite{witherspoon_hochschild_2019}) of an algebra $A$ are defined by considering the functors $\ext$ and $\tor$ in the category of $A$-bimodules. 

\begin{defi}
Let $A$ be an algebra and $M$ an $A$-bimodule. The \textit{Hochschild cohomology} spaces of $A$ with coefficient in $M$ (resp. \textit{Hochschild homology} spaces of $A$ with coefficient in $M$) are the vector spaces \[\HH^\bullet(A,M) := \ext_{A^e}^\bullet(A,M), \quad \HH_\bullet(A,M) := \tor^{A^e}_\bullet(M,A).\]. 
\end{defi}


The general homological finiteness notions defined in the previous section specify to the case of bimodules which is our case of interest when considering Hochschild (co)homology.

\begin{defi}
    Let $A$ be an algebra.
\begin{enumerate}
    \item The \textit{cohomological dimension} of $A$ is defined to be
$\cd(A) := \pdim_{A^e}(A)$.
\item The algebra $A$ is said to be \textit{homologically smooth} if $A$ is of type $FP$ as a left $A^e$-module.
\end{enumerate}
\end{defi}

Notice that if $A$ is a homologically smooth algebra, then $\cd(A) = \max \left\{n\in \mathbb{N},~ \HH^n(A,A^e) \neq 0\right\}$, by Proposition \ref{prop:FP-FPinfty}.

For some (smooth) algebras, appropriate conditions ensure that there exists a duality between Hochschild homology and cohomology, analogous to Poincaré duality in algebraic topology.


\begin{defi}\label{def : duality}
Let $A$ be an algebra. We say that $A$ has \textit{homological duality in dimension} $d$ if :
\begin{enumerate}[label = (\roman*)]
    \item $A$ is a Poincaré duality $A^e$-module of dimension $d$, i.e. $A$ is homologically smooth with $d=\cd(A)$ and $\HH^{i}(A,A^{e})=\{0\}$ for $i\not=d$;
    \item  $D = \HH^d(A,A^e)$, endowed with the $A$-bimodule structure induced by right multiplication of $A^e$, is flat as a left and as a right $A$-module.
\end{enumerate}
If these conditions hold, we say that the bimodule $D$ is a dualizing bimodule of $A$.
\end{defi}

The following result justifies the "homological duality" terminology above. 

\begin{thm}\label{thm : hochschild duality}
Let $A$ be an algebra. If $A$ has homological duality in dimension $d$ with dualizing bimodule $D$ , then for any $A$-bimodule $M$, we have functorial isomorphisms $$\HH^{n}(A,M) \simeq \HH_{d-n}(A,M \otimes_A D )\simeq \HH_{d-n}(A, D\otimes_A M).$$ 
\end{thm}

\begin{proof}
We know from Theorem \ref{thm:hdm} that for any $A$-bimodule $M$ we have 
\[\HH^{i}(A,M)\simeq \tor_{d-i}^{A^{e}}(D,M)
\]
Using that $D$ is assumed to be flat as a left and as a right $A$-module and that $\tor$ can be computed by using resolutions by flat modules, it is not difficult to check that the above spaces on the right are isomorphic to the announced Hochschild homology spaces.
\end{proof}

Definition \ref{def : duality} and Theorem \ref{thm : hochschild duality} include a number of well-studied particular homological dualities in the literature \cite{van_den_bergh_relation_1998, ginzburg_calabi-yau_nodate,LiWaWu14}.

\begin{defi}
Let $A$ be an algebra having homological duality in dimension $d$ and dualizing bimodule $D$.
\begin{enumerate}
    \item $A$ is said to have Van den Bergh duality if the $A$-bimodule $D$ is invertible;
    \item $A$ is said to be twisted Calabi-Yau if $D \simeq A_\sigma$ as an $A$-bimodule, with $\sigma$ an algebra automorphism of $A$. In this case, the automorphism $\sigma$ is called \textit{the Nakayama automorphism} of $A$, and is unique up to inner automorphisms;
    \item $A$ is said to be Calabi-Yau if $D \simeq A$ as $A$-bimodules.
\end{enumerate}
\end{defi}

\begin{rem}
Let $A$ be an algebra and assume that $A$ is a Poincaré duality $A^{e}$-module of dimension $d$. If $D=\HH^d(A,A^{e})$ is flat as a left $A$-module, then by \cite[Theorem 1]{kowalzig_duality_2010}, we have natural isomorphisms for any $A$-bimodule $M$  
$$\HH^{n}(A,M) \simeq \HH_{d-n}(A,M \otimes_A D )$$
Similarly if $D$ is flat as a right $A$-module, then by \cite[Theorem 4]{krahmer_hochschild_2012}, we have natural isomorphisms   
$$\HH^{n}(A,M) \simeq \HH_{d-n}(A,D \otimes_A M )$$
These two results indicate that, although quite flexible, the framework for homological duality in Definition \ref{def : duality} might not be the most general one. We have adopted Definition \ref{def : duality} because we did not see any reason to privilegiate left or right flatness for the dualizing bimodule $D$. 
\end{rem}

See \cite[Section 4.3]{witherspoon_hochschild_2019} and references therein for examples of algebras with homological duality. In particular, algebras of polynomials in $m$ indeterminates have homological duality in dimension $m$.

\subsection{Cohomology of Hopf algebras}

In this subsection we recall some special features of Hochschild (co)homology for Hopf algebras. We begin by recalling Hochschild cohomology and homology spaces of a Hopf algebra $H$ can be described by using suitable $\ext$ and $\tor$ spaces in the category of right or left $H$-modules and resolutions of the trivial module (see \cite{FengTsygan91,ginzburg_cohomology_1993,brown_dualising_2008, wang_calabi-yau_2017}):



\begin{thm}\label{ext hopf}
    Let $H$ be a Hopf algebra and $M$ an $H$-bimodule. We have
    \[
        \HH^n(H,M) \simeq \ext^n_{H^{op}}(k_\varepsilon,\widetilde{M})\simeq \ext_H^n({}_\varepsilon k,\overline{M}), \quad
        \HH_n(H,M) \simeq \tor^H_n(k_\varepsilon,M') \simeq \tor^H_n(M^{''}, {_\varepsilon k})
    \]
    where the $H$-modules $\widetilde{M}$, $\overline{M}$, $M'$, $M''$ all have $M$ as underlying vector space, and $H$-action given, for $m\in N$ and $h\in H$, by the following formulas.
    \begin{enumerate}[label= $\bullet$]
        \item $\widetilde{M}$ has the right $H$-module structure given by $m\cdot h = S(h_{(1)}) m h_{(2)}$;
        \item $\overline{M}$ has the left $H$-module structure given by $h\cdot m = h_{(1)} m S(h_{(2)})$; 
        \item $M'$ has the left $H$-module structure given by $h\rightarrow m = h_{(2)}mS(h_{(1)})$;
        \item $M^{''}$ has  the right $H$-module structure given by $m\leftarrow h = S(h_{2}) m  h_{(1)}$.
    \end{enumerate}
\end{thm}



Homological smoothness for Hopf algebras can be detected using the trivial left or right $H$-module (see \cite[Proposition A.2]{wang_calabi-yau_2017}):

\begin{thm}\label{smooth Hopf}
     Let $H$ be a Hopf algebra. The following assertions are equivalent: 
     \begin{enumerate}[label= (\roman*)]
         \item The algebra $H$ is homologically smooth;
         \item the right $H$-module $k_\varepsilon$ is of type $FP$;
         \item the left $H$-module ${}_\varepsilon k$ is of type $FP$.
     \end{enumerate}
     \end{thm}

The next result, which specializes the isomorphisms of Theorem \ref{ext hopf} in the case $M = H^e$,  is useful in view of a simple characterization of homological duality in the Hopf algebra case. 

\begin{prop}{\rm(\cite[Proposition 2.1.3]{wang_calabi-yau_2017})}\label{iso H}
    Let $H$ be a homologically smooth Hopf algebra. We have
    $$\HH^n(H,H^e) = \ext^n_{H^e}(H,H^e) \simeq \ext^n_{H^{op}}(k_\varepsilon,H) \otimes H \simeq \ext_H^n({}_\varepsilon k,H) \otimes H$$ as right $H^e$-modules for $n \geq 0$ where the $H^e$-module structures on $\ext^n_{H^{op}}(k_\varepsilon,H) \otimes H$ and $\ext_H^n({}_\varepsilon k,H) \otimes H$ are respectively given by $$(x\otimes h) \leftarrow (a\otimes b) = b_{(2)} x \otimes S^2(b_{(1)})ha \quad \text{ and } \quad (y\otimes h) \leftarrow (a\otimes b) =  y a_{(1)} \otimes bhS^2(a_{(2)}),$$ for $a,b,h \in H, \ x\in  \ext^n_{H^{op}}(k_\varepsilon,H) \otimes H$ and $y \in \ext_H^n({}_\varepsilon k,H)$.
\end{prop}


We thus see from the above result that $\HH^n(H,H^e)$ is free  
as a left and as a right $H$-module,  
hence in particular is flat as  a left and as a right $A$-module 
We thus obtain the following characterization of homological duality in the case of Hopf algebras.

\begin{thm}\label{duality Hopf}
Let  $H$ be a   Hopf algebra, and let $d\geq0$. The following assertions are equivalent.
\begin{enumerate}
    \item $H$ has homological duality in dimension $d$;
    \item $H$ is homologically smooth and $\ext_{H}^i( {}_\varepsilon k, H) = 0$  for $i\neq d$;
    \item $H$ is homologically smooth and $\ext_{H^{op}}^i( k_\varepsilon, H) = 0$  for $i\neq d$
    \item The $H$-module ${}_\varepsilon k$ is a Poincaré duality module of dimension $d$.
    \item The $H^{\rm op}$-module  $k_\varepsilon$ is a Poincaré duality module of dimension $d$.
\end{enumerate}
\end{thm}

\begin{rem}
The above result recovers, when $H$ is a group algebra, the homological duality theory for groups due to Bieri-Eckman \cite{bieri_groups_1973}.
\end{rem}

The twisted Calabi-Yau property also has a particularly convenient formulation in the case of Hopf algebras.
The first part of the following result is  \cite[Lemma 2.15]{van_oystaeayen_cleft_2016}, while the second part is a direct consequence of Proposition \ref{iso H}.

\begin{prop}\label{Prop : CY Hopf}
    Let $H$ be a Hopf algebra with homological duality in dimension $d$.
    
    If $H$ is twisted Calabi-Yau with Nakayama automorphism $\sigma$, then $\ext^d_H({}_\varepsilon k,H) \simeq k_\xi$ (resp. $\ext^d_{H^{op}}(k_\varepsilon,H) \simeq {}_\eta k$) as right (resp. left) $H$-modules where $\xi = \varepsilon\circ \sigma$ (resp. $\eta = \varepsilon\circ \sigma^{-1}$). 

    Conversely, if  $\ext^d_H({}_\varepsilon k,H) \simeq k_\xi$ (resp. $\ext^d_{H^{op}}(k_\varepsilon,H) \simeq {}_\eta k$) as right (resp. left) $H$-modules then $H$ is twisted Calabi-Yau with Nakayama automorphism given by $\sigma (h) = \xi(h_{(1)}) S^2(h_{(2)})$ (resp. $\sigma (h) = \eta(h_{(2)}) S^2(h_{(1)})$) for $h\in H$.
\end{prop}

We refer the reader to \cite{brown_dualising_2008} for various examples of twisted Calabi-Yau Hopf algebras.

\section{Exact sequences of Hopf algebras}

\subsection{Definitions}

\begin{defi}\label{def : exact sequence}
A sequence of Hopf algebra maps 
\[k \longrightarrow B \overset{i}{\longrightarrow} A \overset{p}{\longrightarrow} H \longrightarrow k\]
is said to be exact \cite{andruskiewitsch_extensions_1996} if the following conditions hold:

\begin{enumerate}[label = (\arabic*), parsep=0cm,itemsep=0.1cm,topsep=0cm]
    \item $i$ is injective and $p$ is surjective,
    \item $\Ker(p) = i(B)^+ A = A i(B)^+$,
    \item $i(B) = A^{co H} = {^{co H}A}$
\end{enumerate}
where $A^{co H}= \{a \in A,~ a_{(1)}\otimes p(a_{(2)}) = a\otimes 1\}$ and ${}^{co H}A= \{a \in A,~ p(a_{(1)})\otimes a_{(2)} = 1\otimes a\}$.
\end{defi}

\begin{ex}
    Exact sequences of groups correspond, as expected, to exact sequences of their group algebras.
\end{ex}

\begin{ex}
    Let $\Gamma$ be a discrete group acting on a Hopf algebra $A$ \textit{via} a group morphism $\alpha : \Gamma \to {\rm Aut}(A)$. Recall that the crossed product $A \rtimes k \Gamma$ has $A \otimes k\Gamma$ as underlying vector space and product defined by $$(a\otimes g) \cdot (b\otimes h) = a \alpha_g(b) \otimes gh \quad \text{ for } \quad  a,b\in A \quad \text{ and } \quad g,h\in\Gamma.$$
    This algebra is naturally endowed with a Hopf algebra structure defined by $$\Delta (a\otimes g) = a_{(1)} \otimes g \otimes a_{(2)} \otimes g, \quad \varepsilon(a\otimes g) = \varepsilon(a), \quad S(a\otimes g) = \alpha_{g^{-1}}\bigl(S(a)\bigr) \otimes g^{-1}$$ for any $a\in A$ and $g\in \Gamma$. 
    Then $k\longrightarrow A \overset{\id \otimes u}{\longrightarrow} A \rtimes k \Gamma \overset{\varepsilon \otimes \id}{\longrightarrow} k\Gamma \longrightarrow k$ is an exact sequence of Hopf algebras.
\end{ex}

\begin{rem}\label{rem:exact}
There are a number of situations where some of the axioms of an exact sequence of Hopf algebras automatically hold.   Let $A$ be a Hopf algebra with bijective antipode.
\begin{enumerate}
    \item If $B\subset A$ is a Hopf subalgebra such that $B^+A=AB^+$, then $B^+A$ is a Hopf ideal, and if $A$ is faithfully flat as a left or right $B$-module, then the sequence $k \to B \to A \to A/B^+A\to k$ is exact, by (for example) \cite[Lemma 1.3]{schneider92}. Notice that faithful flatness automatically holds if $A$ is cosemisimple \cite{chi14}.
    \item If $p: A \to H$ is a surjective Hopf algebra map  such that $A^{co H} = {^{co H}A}$, then if $p : A \to H$ is faithfully coflat, then $k \to A^{co H} \to A \to H\to k$ is an exact sequence of Hopf algebras, by \cite[Theorem 2]{tak79} combined with \cite[Remark 1.3]{mulsch}. Notice that $p: A \to H$ is automatically faithfully coflat if $H$ is cosemisimple. 
\end{enumerate}
\end{rem}

We will extensively use the following well-known fact.

\begin{lemma}\label{lem : B stable}
    Let $k \longrightarrow B \overset{i}{\longrightarrow} A \overset{p}{\longrightarrow} H \longrightarrow k$ be an exact sequence of Hopf algebras, then $B$ is stable under the adjoint actions of $A$ : for $b\in B$ and $a\in A$, we have $a_{(1)}bS(a_{(2)}) \in B$ and $S(a_{(1)})b a_{(2)} \in B$. If moreover $A$ has bijective antipode, we have as well $a_{(2)}bS^{-1}(a_{(1)}) \in B$ and $S^{-1}(a_{(2)})b a_{(1)} \in B$.
\end{lemma}

\begin{proof}
    For $b \in B$ and $a\in A$, we have 
    \[\begin{aligned}
        \bigl(a_{(1)}bS(a_{(2)})\bigr)_{(1)} \otimes p\bigl(\bigl(a_{(1)}bS(a_{(2)})\bigr)_{(2)}\bigr) &= a_{(1)}b_{(1)}S(a_{(3)})_{(1)}\otimes p(a_{(2)})p(b_{(2)})p(S(a_{(3)})_{(2)}) \\ 
        &= a_{(1)}bS(a_{(4)})\otimes p(a_{(2)}S(a_{(3)}))\\ 
        &= a_{(1)}bS(a_{(3)})\otimes \varepsilon(a_{(2)})
        = a_{(1)}bS(a_{(2)}) \otimes 1
    \end{aligned}\]
    hence $a_{(1)}bS(a_{(2)}) \in A^{co H} = B$. 
    We also have \[\begin{aligned}
        p\bigl(\bigl(S(a_{(1)})b a_{(2)})\bigr)_{(1)}\bigr) \otimes \bigl(S(a_{(1)})b a_{(2)}\bigr)_{(2)} &= p\bigl(S(a_{(1)})_{(1)}\bigr) p(b_{(1)}) p(a_{(2)}) \otimes S(a_{(1)})_{(2)}b_{(2)} a_{(3)}\\
        &= p\bigl(S(a_{(2)})\bigr) p(a_{(3)}) \otimes S(a_{(1)})b a_{(4)}\\
        &= \varepsilon(a_{(2)}) \otimes S(a_{(1)})ba_{(3)}
        = 1 \otimes S(a_{(1)})b a_{(2)}
    \end{aligned}\]
  and  hence $S(a_{(1)})b a_{(2)} \in {}^{co H}A = B$. The last assertions follow from the first ones if $B$ is $S^{-1}$-stable, or can be proved directly without this assumption, with similar arguments as above. This is left to the reader.
\end{proof}

\subsection{Spectral sequences for exact sequences of Hopf algebras}\label{Sec:Spec Seq}

Let $k \longrightarrow B \overset{i}{\longrightarrow} A \overset{p}{\longrightarrow} H \longrightarrow k$ be an exact sequence of Hopf algebras. In this subsection we construct spectral sequences connecting the (co)homology groups $\ext^\bullet_{A}({}_\varepsilon k,-)$ (resp. $\tor_\bullet^A(k_\varepsilon , -)$) with $\ext^\bullet_{B}({}_\varepsilon k,-)$ and $\ext^\bullet_{H}({}_\varepsilon k,-)$ (resp. $\tor_\bullet^B(k_\varepsilon , -)$ and $\tor_\bullet^H(k_\varepsilon , -)$). These spectral sequences generalize the classical Lyndon-Hochschild-Serre spectral sequences and can be deduced from Stefan's spectral sequences \cite{stefan_hochschild_1995} in the setting of Hopf-Galois extensions.  However we work them in detail, on the one hand for the sake of completeness, and on the other hand because there are some specific technicalities  (notably regarding certain $H$-actions) that are needed in the proof of Theorem \ref{th seq}

Following the ideas in \cite{stefan_hochschild_1995}, we first need to define a structure of $H$-module on the cohomology of $B$. 
In order to do so, we will use resolutions of the $A$-module ${}_\varepsilon k$ to compute the homology spaces $\ext_B^q({}_\varepsilon k,M)$ for a left $A$-module $M$, which is allowed thanks to the following well-known result.

\begin{lemma}
    Let $A$ be a Hopf algebra and $B \subset A$ a Hopf subalgebra such that $A$ is projective as a left $B$-module. If $P_\bullet $ is a resolution of the left $A$-module ${}_\varepsilon k$, then the restriction of $P_\bullet $ is a resolution of the left $B$-module ${}_\varepsilon k$.
\end{lemma}

We first define the $H$-module structure on $\ext^0_{B}({}_\varepsilon k,-)\simeq M^B = \left\{ m \in M,\ bm = \varepsilon(b) m, \ b\in B\right\}$.

\begin{lemma}\label{lemma:structure}
    Let $k \longrightarrow B \overset{i}{\longrightarrow} A \overset{p}{\longrightarrow} H \longrightarrow k$ be an exact sequence of Hopf algebras such that $A$ is flat as a left and right $B$-module and let $M$ be a left $A$-module.
    There is a left $H$-module structure  $\cdot$ on $\ext_B^0({}_\varepsilon k,M)$, defined by $$p(a) \cdot m = a \cdot m $$ for any $a\in A$ and $m\in \ext_B^0({}_\varepsilon k,M) \simeq M^B$.
\end{lemma}

\begin{proof}
    We have to check that the above formula is well defined. 
    Consider $a,\ a' \in A$ such that $p(a) = p(a')$: there exist $b_1,\ldots , b_m\in B^+$ and $a_1, \ldots a_m \in A$ such that $a' - a =  \sum\limits_ia_ib_i $.
    Then for $m\in \ext_B^0({}_\varepsilon k,M) \simeq M^B$, we have $a' \cdot m = (a +  \sum\limits_ia_i b_i ) \cdot m =  a \cdot m$ since $m \in M^B$. This proves that  our expression does not depend on the choice of $a\in A$.

    For $m \in M^B,\ a \in A$ and $b\in B$, we have $$(b-\varepsilon(b)) \cdot (a \cdot m)  = \bigl((b-\varepsilon(b))a\bigr) \cdot m$$ with $(b-\varepsilon(b))a \in  B^+A = AB^+$, hence $(b-\varepsilon(b))\cdot (a\cdot m) = 0$ and $a\cdot m \in M^B$. We conclude that  the above formula indeed defines an action of $H$ on $\ext^0_B({}_\varepsilon k ,M)$.
\end{proof}

We recall some terminology and a few key results used in \cite{stefan_hochschild_1995} in order to extend the $H$-module structure defined on $\ext_B^0({}_\varepsilon k,M)$ to $\ext_B^q({}_\varepsilon k,M)$ for any $q\geq 1$.

\begin{defi}
Let $\mathfrak{C}$ and $\mathfrak{D}$ be two abelian categories with enough injectives and let $T_\bullet = (T_n)_{n\in \mathbb{Z}}$ be a sequence of covariant functors, $T_n : \mathfrak{C} \to \mathfrak{D}$. Then  $T_\bullet$ is said to be a \textit{a $\delta$-functor} if for any short exact sequence $$0\to M' \to M \to M''\to 0$$ in $\mathfrak{C}$, there are natural connecting homomorphism $$\delta :T_n(M'') \to T_{n+1}(M')$$ and all compositions are zero in the sequence $$T_n(M') \to T_n(M) \to T_n(M'') \overset{\delta}{\to}T_{n+1}(M').$$

If in addition $T_n = 0$ for $n<0$ and the above sequence is exact, then $T_\bullet$ is said to be a \textit{cohomological functor}.
\end{defi}

\begin{defi}
If $T_\bullet= (T_n)$ and $S_\bullet= (S_n)$ are $\delta$-functors, a morphism of $\delta$-functors from $S_\bullet$ to $T_\bullet$ is a family $\varphi = (\varphi_n)_{n\in\mathbb{Z}}$ of natural transformations $\varphi_n : S_n \to T_n$ such that the square 
\[\begin{tikzcd}    S_n(M'') \arrow[d,"\varphi_n(M'')"'] \arrow[r,"\delta_S"] & S_{n+1}(M') \arrow[d,"\varphi_{n+1}(M')"]\\
    T_n(M'') \arrow[r,"\delta_T"'] & T_{n+1}(M')\\
\end{tikzcd}\]
commutes for every exact sequence $0 \to M' \to M \to M'' \to 0$ in $\mathfrak{C}$ and every $n\in \mathbb{Z}$.
\end{defi}

\begin{defi}
A cohomological functor $T_\bullet=(T_n) : \mathfrak{C} \to \mathfrak{D}$ is \textit{coeffaceable} if, for each $n>0$, any object $M$ in $\mathfrak{C}$ can be embedded in an object $N$ such that $T_n(N) = 0$.
\end{defi}

\begin{thm}\label{delta functors}
\cite[Theorem 7.5]{brown_cohomology_1982} Let $T_\bullet$ be a cohomological and coeffaceable functor. If $S_\bullet$ is a $\delta$-functor and $\varphi_0 : T_0 \to S_0$ is a natural transformation, then $\varphi_0$ extends uniquely to a map of $\delta$-functors $\varphi_\bullet : T_\bullet \to S_\bullet$.
\end{thm}

\begin{prop}\label{Prop:structure H mod}
Let $k \longrightarrow B \overset{i}{\longrightarrow} A \overset{p}{\longrightarrow} H \longrightarrow k$ be an exact sequence of Hopf algebras such that $A$ is flat as a left and right $B$-module and let $M$ be a left $A$-module.
There is a left $H$-action on $\ext_B^q({}_\varepsilon k,M)$ for $q\geq 0$ which extends the structure on $\ext_B^0({}_\varepsilon k,M)$ defined in Lemma \ref{lemma:structure}.
\end{prop}

\begin{proof}
    We first check that if $R : {}_A\mathcal{M} \to {}_B\mathcal{M}$ denotes the functor of restriction of scalars, then the cohomological functor $\ext_B^\bullet({}_\varepsilon k,-)\circ R: {}_{A}\mathcal{M}\to k-\Vect$ is coeffaceable. 

    Let $M$ be an $A$-module, there exists $N$ an injective $A$-module such that $M \subset N$. The functor $R$ has $A\otimes_B -$ as left adjoint, hence we have $\Hom_B(-,R(N)) \simeq \Hom_A(A\otimes_B-_, N)$. The right $B$-module $A$ is flat and $N$ is injective hence $\Hom_A(A\otimes_B -, N)$ is exact and $R(N)$ is an injective $B$-module. Thus, we have $\ext_B^n({}_\varepsilon k,-)\circ R (N) = 0$ for $n\geq 1$ and $\ext_B^\bullet({}_\varepsilon k,-)\circ R$ is coeffaceable.

    For any $h \in H$ let $\rho_0^h(-) : \ext_B^0({}_\varepsilon k,-)\circ R \to \ext_B^0({}_\varepsilon k,-)\circ R$ be the natural morphism
    given by  $\rho_0^h(M)(m) = h \cdot m$, for any $A$-module $M$, $m\in \ext_B^0({}_\varepsilon k,M)$ and $h\in H$. By Theorem \ref{delta functors} there is a unique morphism of $\delta$-functors $\rho_\bullet^h(-) : \ext_B^\bullet({}_\varepsilon k,-)\circ R \to \ext_B^\bullet({}_\varepsilon k,-)\circ R$ which extends $ \rho_0^h(-)$. 
We have, for any $h,k\in H$, $\rho_0^{hk}(-) = \rho_0^{h}(-) \circ \rho_0^{k}(-) $ hence uniqueness in Theorem \ref{delta functors} also gives that$\rho_\bullet^{hk}(-) = \rho_\bullet^{h}(-) \circ \rho_\bullet^{k}(-)$, and $\rho_\bullet^h(-)$ defines a structure of $H$-module.
\end{proof}

\begin{lemma}\label{acyclicity}
    Let $k \longrightarrow B \overset{i}{\longrightarrow} A \overset{p}{\longrightarrow} H \longrightarrow k$ be an exact sequence of Hopf algebras and let $M$ be an injective left $A$-module. Then $$\ext_{H}^q({}_\varepsilon k,M^B) = 0\ {\rm for } \ q \geq 1.$$
\end{lemma}

\begin{proof}
    Let $i : M \to \Hom_k(A, M),\ i(m) (x) = xm$. Then $i$ is $A$-linear and injective, hence the injectivity of $M$ ensure that there is an $A$-module $N$, such that $\Hom_k(A, M) \simeq M \oplus N$ as $A$-modules. 
     Hence $\Hom_k(A, M)^B \simeq M^B \oplus N^B$ as $H$-modules and since $\ext$ preserves direct sums, it is sufficient to show that $$\ext_H^q({}_\varepsilon k,\Hom_k(A, M)^B) = 0\ {\rm for } \ q\geq 1.$$
    Moreover, we have an isomorphism of right $H$-modules $\varphi : \Hom_k(A, M)^B \to \Hom_k(H, M)$ given by $\varphi(f)(p(a)) = f(a)$ for $f \in \Hom_k(A,M)^B$ and $p(a) \in H$.

    Now, consider $\mathbf{P^\bullet}$ a resolution of ${}_\varepsilon k$ by projective $H$-modules. Then $\ext_H^\bullet({}_\varepsilon k,\Hom_k(A, M)^B)$ is the cohomology of the complex $$\Hom_H(\mathbf{P^\bullet},\Hom_k(A,M)^B) \simeq \Hom_H(\mathbf{P^\bullet},\Hom_k(H,M))\simeq \Hom_k(\mathbf{P^\bullet},M).$$ 
    The exactness of $\mathbf{P^\bullet}$ and the fact that $k$ is a field gives that this complex is exact, and hence the result. 
\end{proof}

We now get the following spectral sequence using the Grothendieck spectral sequence.

\begin{thm}\label{thm : ss ext}
Let $k \longrightarrow B \overset{i}{\longrightarrow} A \overset{p}{\longrightarrow} H \longrightarrow k$ be an exact sequence of Hopf algebras and let $M$ be a left $A$-module.
 Assume that $A$ is flat as left and right $B$-module. Then there is a spectral sequence 
\[E^{p,q}_2 = \ext_{H}^p({}_\varepsilon k,\ext_{B}^q({}_\varepsilon k,M)) \implies \ext^{p+q}_{A}({}_\varepsilon k,M)\]
which is natural in M. 
\end{thm}

\begin{proof}
    We proceed as in \cite{stefan_hochschild_1995}. Let $F,\ F_1, \ F_2$ be the following three functors 
    \begin{align*}
        F : {}_A\mathcal{M} \to {}_k\mathcal{M},\quad F = \Hom_A({}_\varepsilon k,-) = (-)^A,\\
        F_1 : {}_H\mathcal{M} \to {}_k\mathcal{M},\quad F_1 = \Hom_H({}_\varepsilon k,-) = (-)^H,\\
        F_2 : {}_A\mathcal{M} \to {}_H\mathcal{M},\quad F_2 = \Hom_B({}_\varepsilon k,-) = (-)^B.
    \end{align*}
For any $A$-module $M$, we have 
\begin{align*}
    F_1\circ F_2 (M) = (M^B)^H  &= \biggl(\{m\in M,\ b\cdot m = \varepsilon(b) m,\ \forall b\in B\}\biggr)^H\\ 
    &= \{m\in M,\ b\cdot m = \varepsilon(b)m,\ \forall b\in B, \ h\cdot m = \varepsilon(h)m,\ \forall h \in H\}\\
    &= \{m\in M,\ b\cdot m = \varepsilon(b)m,\ \forall b\in B, \ p(a)\cdot m = \varepsilon(p(a))m,\ \forall a \in A\}\\
    &= \{m\in M,\ a\cdot m = \varepsilon(a)m,\ \forall a\in A\}\\
    &= F(M),
\end{align*}
 hence $F_1\circ F_2 = F$. Moreover, if $M$ is an injective $A$-module, Proposition \ref{acyclicity} gives that $F_2(M)$ is $F_1$-acyclic.
    Therefore, applying the Grothendieck spectral sequence (\cite[Theorem 5.8.3]{weibel_introduction_2008}), we get the desired spectral sequence.
\end{proof}

The $H$-module structure involved in the spectral sequence in Theorem \ref{thm : ss ext} is, by construction of the Grothendieck spectral sequence, the one in Proposition \ref{Prop:structure H mod}. 
In Section \ref{sec : duality ex seq H}, we will need a more explicit description of this $H$-module structure on $\ext_B^\bullet({}_\varepsilon k,-)$, that the following result provides.

\begin{prop}\label{structure 2}
Let $k \longrightarrow B \overset{i}{\longrightarrow} A \overset{p}{\longrightarrow} H \longrightarrow k$ be an exact sequence of Hopf algebras such that $A$ is flat as a left and right $B$-module and let $M$ be a left $A$-module.
Let $P_\bullet \to {}_\varepsilon k$ be a projective resolution of ${}_\varepsilon k$ by $A$-modules. For $q\geq 0$, $a\in A, \ f \in \Hom_B(P_q,M)$ and $x\in P_q$ the expression \begin{align} (p(a) \cdot f) (x) = a_{(1)}f(S(a_{(2)})x) \tag{$*$} \end{align} 
defines an $H$-module structure on $\Hom_B(P_q,M)$ and on the complex $\Hom_B(P^\bullet, M)$, and hence induces an $H$-module structure on $\ext_B^q({}_\varepsilon k,M)$, which coincides with the $H$-module structure of Proposition \ref{Prop:structure H mod}.
\end{prop}

\begin{proof}
    First, we have to check that this defines an $H$-action. Consider $a,\ a' \in A$ such that $p(a) = p(a')$: there exist $b_1,\ldots , b_m\in B^+$ and $a_1, \ldots a_m \in A$ such that $a' - a =  \sum\limits_ia_ib_i $. For $f\in \Hom_B(P_q,M)$ and $x\in P_q$, we have $$\sum\limits_i a_{i(1)} b_{i(1)}f(S(b_{i(2)})S(a_{i(2)})x) = \sum\limits_i a_{i(1)} b_{i(1)} S(b_{i(2)})f(S(a_{i(2)})x) = \sum\limits_i a_{i(1)} \varepsilon(b_i) f(S(a_{i(2)})x) = 0 .$$
    Hence $(*)$ does not depend on the choice of $a \in A$ such that $h=p(a)$.
Now, for $f \in \Hom_B(P_q,M)$, $a\in A$,  $b\in B$ and $x \in P_q$, using Lemma \ref{lem : B stable} we have
    \begin{align*}
        \bigl(p(a)\cdot f\bigr)(bx) &= a_{(1)}f(S(a_{(2)})bx)
        = a_{(1)}f(S(a_{(2)})b \varepsilon(a_{(3)})x)\\
        &= a_{(1)}f(S(a_{(2)})b a_{(3)} S(a_{(4)})x)
        = a_{(1)}S(a_{(2)})b a_{(3)}f( S(a_{(4)})x)\\
        &= b a_{(1)}f(S(a_{(2)})x) = b \bigl(p(a)\cdot f\bigr).
    \end{align*}
    Thus $p(a) \cdot f$ is $B$-linear. We obtain that $(*)$ is well-defined and one easily checks that this defines a structure of left $H$-module on $\Hom_B(P_q,M)$, which is compatible with the differential of the complex $\Hom_B(P^\bullet, M)$, and hence induces an $H$-module structure on $\ext_B^q({}_\varepsilon k,M)$.
    
    For any $h = p(a) \in H$ let $\psi_q^h(-) : \ext_B^q({}_\varepsilon k,-)\circ R \to \ext_B^q({}_\varepsilon k,-)\circ R$ be the natural morphism given by  $\psi_q^h(M)(f) = p(a) \cdot f$, for any $A$-module $M$ and any $f\in \hom_B(P_q,M)$. 
    We will check that $\psi_\bullet^h(-) : \ext_B^q({}_\varepsilon k,-)\circ R \to \ext_B^q({}_\varepsilon k,-)\circ R$ is a morphism of $\delta$-functors and that $\psi_0^h(-) = \rho_0^h(-)$, hence the unicity statement in Theorem \ref{delta functors} will conclude the proof.

    We have $\ext_B^0({}_\varepsilon k,M) \simeq M^B$ and this isomorphism is given by $[f] \mapsto \widetilde{f}(1)$ for any $[f] \in \ext_B^0({}_\varepsilon k,M)$ where $\widetilde{f} \in \Hom_B({}_\varepsilon k,M)$ is the map induced by $f$. It is clear that this isomorphism sends $p(a) \cdot f = \bigl( x \mapsto a_{(1)} f(S(a_{(2)})x) \bigr)$ to $a \widetilde{f}(1) $ for any $a\in A$ and $[f]\in \ext_B^0({}_\varepsilon k,M)$ hence, we have $\psi_0^h(-) = \rho_0^h(-)$ for any $h = p(a) \in H$.

    In order to prove that $\psi_\bullet^h$ is a morphism of $\delta$-functors, we first explain de definition of the connecting homomorphism $\delta$ in our case of interest. Let $P_\bullet \to {}_\varepsilon k$ be a projective resolution of the left $A$-module ${}_\varepsilon k$ with $\partial_q : P_q \to P_{q-1}$ the morphisms of this resolution, and $d_q = -\circ \partial_q$. Let $0\to M' \to M \overset{\pi}{\to} M'' \to 0$ be an exact sequence of $A$-modules. For $q\geq 0$ and $[\phi] \in \ext_B^q({}_\varepsilon k,M'')$, let $\widetilde{\phi} : P_q \to M$ such that $\phi = \pi\circ \widetilde{\phi}$. Since $\phi$ is a cocycle, we then have $d_{q+1}(\widetilde{\phi})(P_{q+1})\subset M'$, and we set 
    $\delta([\phi]) = [d(\widetilde{\phi})]$. 

    For $h = p(a) \in H,\ q\geq 0, \ \alpha \in \ext_B^q({}_\varepsilon k, M''), \ \phi : P_q \to M''$ a representative of $\alpha$, $\widetilde{\phi} : P_q \to M$ such that $\phi = \pi \circ \widetilde{\phi}$ and $x\in P_{q+1}$, we have    
\begin{align*}
    p(a) \cdot d_{q+1}(\widetilde{\phi}) (x) &= p(a) \cdot \bigl(\widetilde{\phi} \circ \partial_{q+1}\bigr)(x) \\
    &= a_{(1)} \bigl(\widetilde{\phi} \circ \partial_{q+1}\bigr)(S(a_{(2)})x)\\
    &= a_{(1)} \widetilde{\phi}(\partial_{q+1}(S(a_{(2)})x))\\
    &= a_{(1)} \widetilde{\phi}(S(a_{(2)})\partial_{q+1}(x)) \quad {\rm since } \ \partial_{q+1} \ {\rm is} \ A{\rm -linear}\\
    &= \bigl(p(a)\cdot \widetilde{\phi}\bigr)(\partial_{q+1}(x))\\
    &=d_{q+1}\bigl(p(a) \cdot \widetilde{\phi}\bigr)(x).
\end{align*}
    We have, since $\pi$ is $A$-linear, $\pi(p(a) \cdot \widetilde{\phi}) = p(a) \cdot \phi$, hence $d_{q+1}\bigl(p(a) \cdot \widetilde{\phi}\bigr)$ is a representative of $\delta\bigl(p(a) \cdot [\phi]\bigr) = \delta(p(a) \cdot \alpha)$. One the other hand $p(a) \cdot d_{q+1}(\widetilde{\phi})$ is a representative of $p(a)\cdot [d_{q+1}(\widetilde{\phi})] = p(a) \cdot \delta([\phi]) = p(a) \cdot \delta(\alpha)$.
    
    Hence we have $\delta\bigl(\psi_q^h(M'')(\alpha)\bigr)= \psi_{q+1}^h(M')(\delta(\alpha))$, and
 it follows that for any $h\in H$,  $\psi_\bullet^h(-)$ and $\rho_\bullet^h(-)$ are morphisms of $\delta$-functors which both extend $\rho_0^h(-)$, and hence by unicity of such an extension, we have $\psi_\bullet^h(-) = \rho_\bullet^h(-)$, which ends the proof.
\end{proof}

We now turn to the construction of the spectral sequence in homology. Recall that for $M$ a left $A$-module, we have $\tor_0^B(k_\varepsilon,M) \simeq M/[M,B] $ where $[M,B] = {\rm Span}\{bm - \varepsilon(b) m \ | \ b\in B, m\in M\} = B^+M$. We define a structure of left $H$-module on $\tor_0^B(k_\varepsilon,M)$.

\begin{prop}
    Let $k \longrightarrow B \overset{i}{\longrightarrow} A \overset{p}{\longrightarrow} H \longrightarrow k$ be an exact sequence of Hopf algebras. For $M$ a left $A$-module, there is a natural left $H$-module structure on $\tor^B_0(k_\varepsilon,M)$ defined by $$p(a)\cdot \pi_M(m) = \pi_M(a m)$$ for any $a\in A$ and $m\in M$ where $\pi_M$ is the canonical projection of $M$ on $M/B^+M$.
\end{prop}

\begin{proof}
Let  $a,a' \in A$ be such that $p(a) = p(a')$: there exist $b_1,\ldots, b_m \in B^+$ and $a_1,\ldots, a_m \in A$ such that $a-a' =\sum\limits_ia_ib_i $. 
    For $m\in M$, we have $\pi_M((a-a')m) =
    \pi_M(\sum_ia_ib_i m)= \sum\limits_i \varepsilon(a_ib_i) m=0$.
   Hence for $h= p(a) \in H$, the application $\psi_h : M \to \tor^B_0(k_\varepsilon,M), \ m \mapsto \pi_M(am)$, does not depend on the choice of $a$.
    For $h = p(a) \in H, \ b\in B$ and $m \in M$, we have $$\pi_M(a(bm - \varepsilon(b) m)) = \pi_M(a(b-\varepsilon(b))m) = 0.$$ 
    Hence the application $\psi_h$ satisfies $[M,B] \subset \ker(\psi_h)$ and induces the expected linear map $\phi_h : \tor^B_0(k_\varepsilon,M) \to \tor^B_0(k_\varepsilon,M)$, which defines the announced left $H$-module structure on $\tor^B_0(k_\varepsilon,M)$.
\end{proof}

As in the case of cohomology spaces, if $R : {}_A\mathcal{M} \to {}_B\mathcal{M}$ denotes the functor of restriction of scalars, then the functor $\tor^\bullet_B(k_\varepsilon ,-)\circ R$ is homological and effaceable, hence the above structure of $H$-module extends to the homology spaces of higher degree.

\begin{prop}
    Let $k \longrightarrow B \overset{i}{\longrightarrow} A \overset{p}{\longrightarrow} H \longrightarrow k$ be an exact sequence of Hopf algebras and $M$ a left $A$-module. The previous action extends to an action on $\tor_n^B(k_\varepsilon,M)$ for $n\geq 0$.
\end{prop}

\begin{lemma}\label{lemma:iso tor0}
    Let $k \longrightarrow B \overset{i}{\longrightarrow} A \overset{p}{\longrightarrow} H \longrightarrow k$ be an exact sequence of Hopf algebras. For any vector space $V$, we have $\tor_0^B(k_\varepsilon,A\otimes V) \simeq H\otimes V$ as left $H$-modules.
\end{lemma}

\begin{proof}
    We have $\tor_0^B(k_\varepsilon,A\otimes V) \simeq A\otimes V /[A\otimes V,B]$. Consider the map 
    \begin{align*}
    \varphi : A\otimes V \ & \longrightarrow H\otimes V \\
a\otimes v & \longmapsto p(a)\otimes v
\end{align*}
which is clearly surjective since $p$ is. A direct and easy computation shows that $\ker(\varphi) = [A\otimes V,B]$, hence it induces the isomorphism
    \begin{align*}
     A\otimes V /[A\otimes V,B]\ & \longrightarrow H\otimes V \\
\overline{a\otimes v} & \longmapsto p(a)\otimes v
\end{align*}    
which is clearly $H$-linear.
\end{proof}

\begin{prop}\label{prop : acyc homo}
    Let $k \longrightarrow B \overset{i}{\longrightarrow} A \overset{p}{\longrightarrow} H \longrightarrow k$ be an exact sequence of Hopf algebras. Let $M$ be a projective left $A$-module, then $$\tor_n^H(k_\varepsilon, \tor_0^B(k_\varepsilon,M)) =0, \quad {\rm for} \ n\geq1. $$
\end{prop}

\begin{proof}
    The $A$-module $M$ being projective, there exists a module $N$ and a vector space $V$ such that $A \otimes V \simeq M \oplus N$ as $A$-modules. Since $\tor$ preserves direct sums, it is sufficient to prove that $$\tor_n^H(k_\varepsilon, \tor_0^B(k_\varepsilon,A \otimes V)) =0, \quad {\rm for } \ n\geq1. $$

    Now, let $\mathbf{P^\bullet}$ be a projective resolution of the $H$-module $k_\varepsilon$, then $\tor_\bullet^H(k_\varepsilon, \tor_0^B(k_\varepsilon,A \otimes V))$ is the homology of the complex $$\mathbf{P^\bullet} \otimes_H \tor_0^B(k_\varepsilon,A \otimes V) \simeq \mathbf{P^\bullet} \otimes_H (H \otimes V) \simeq \mathbf{P^\bullet} \otimes V$$
    where the first isomorphism is Lemma \ref{lemma:iso tor0}. Finally using the exactness of $\mathbf{P^\bullet}$, we get that this complex is exact since $k$ is a field. Hence $$\Homo_n(\mathbf{P^\bullet} \otimes_H \tor_0^B(k_\varepsilon,A \otimes V)) \simeq \Homo_n(P^\bullet \otimes V) = 0$$ for $n\geq 1$
\end{proof}

\begin{thm}\label{thm : ss tor}
    Let $k \longrightarrow B \overset{i}{\longrightarrow} A \overset{p}{\longrightarrow} H \longrightarrow k$ be an exact sequence of Hopf algebras such that $A$ is projective as a left and right $B$-module and let $M$ be a left $A$-module.
    There is a spectral sequence 
\[E_{p,q}^2 = \tor^{H}_p(k_\varepsilon,\tor^{B}_q(k_{\varepsilon},M)) \implies \tor_{p+q}^{A}(k_{\varepsilon},M)\]
which is natural in M.  
\end{thm}

\begin{proof}
    We reason as in \cite{stefan_hochschild_1995}: let $F,\ F_1, \ F_2$ be the following three functors 
    \begin{align*}
        F : {}_A\mathcal{M} \to {}_k\mathcal{M},\quad F = k_\varepsilon \otimes_A - = -/[-,A],\\
        F_1 : {}_H\mathcal{M} \to {}_k\mathcal{M},\quad F_1 = k_\varepsilon \otimes_H - = -/[-,H],\\
        F_2 : {}_A\mathcal{M} \to {}_H\mathcal{M},\quad F_2 = k_\varepsilon \otimes_B - = -/[-,B].
    \end{align*}
For any $A$-module $M$, we have 
\[
    F_1\circ F_2 (M) = \tor_0^H(k_\varepsilon, M/B^+M) = \bigl(M/B^+M\bigr)/H^+\bigl(M/B^+M\bigr)
\]
Consider the canonic projection $\theta : M/B^+M \to M/A^+M$, we denote $\pi^{(B)}$ (resp. $\pi^{(A)}$) the projection of $M$ onto $M/B^+M$ (resp. $M/A^+M$). For $h=p(a)\in H$ and $m\in M$, we have 

$$\theta\biggl((h-\varepsilon(h))\cdot \pi^{(B)}(m)\biggr) = \theta\biggl(\pi^{(B)}\bigl((a-\varepsilon(a))m\bigr)\biggr) = \pi^{(A)}\bigl((a-\varepsilon(a))m\bigr) = 0.$$ Hence $H^+ \bigl(M/B^+M\bigr) \subset \ker(\theta)$. On the converse, if $\theta\bigl(\pi^{(B)}(m)\bigr) = \pi^{(A)}(m) = 0$, there exist $a_1,\dots, \ a_n \in A^+$ and $m_1,\dots, \ m_n$ such that $m = \sum\limits_i a_i m_i$. We thus have $$\pi^{(A)}(m) = \theta\bigl(\pi^{(B)}(\sum\limits_i a_i m_i)\bigr) = \theta\bigl(\sum\limits_i p(a_i)\cdot \pi^{(B)}(m_i)\bigr)$$ 
with $\sum\limits_i p(a_i)\cdot \pi^{(B)}(m_i) \in H^+ \bigl(M/B^+M\bigr)$, hence we have $\ker(\theta) \bigl(M/B^+M\bigr)$. We get an isomorphism $\widetilde{\theta} : \bigl(M/B^+M\bigr)/H^+\bigl(M/B^+M\bigr) \to M/A^+M$ and finally $F = F_1\circ F_2(M)$.

Moreover, if $M$ is a projective $A$-module, Proposition \ref{prop : acyc homo} gives that $F_2(M)$ is $F_1$-acyclic.
    Hence, applying the Grothendieck spectral sequence in homology (\cite[Theorem 5.8.4]{weibel_introduction_2008}), we get the desired spectral sequence.

\end{proof}

\begin{rem}\label{ff=proj}
It is shown in \cite[Corollary 1.8]{schneider92} that if $A$ is Hopf algebra with bijective antipode and $B\subset A$ is a Hopf subalgebra, then $A$ is faithfully flat as a (left or right) $B$-module if an only if it is projective as a (left or right) $B$-module. Hence, in order to invoke one of the above spectral sequences,  we can indifferently assume projectivity or faithful flatness.
\end{rem}

\section{Homological smoothness and exact sequences of Hopf algebras}\label{Sec:Smoothness}

In this section, using the spectral sequences defined in Section \ref{Sec:Spec Seq}, we prove that if $k \longrightarrow B \overset{i}{\longrightarrow} A \overset{p}{\longrightarrow} H \longrightarrow k$ is an exact sequence of Hopf algebras with bijective antipodes such that $A$ is faithfully flat as a left and right $B$-module, then homological smoothness of $B$ and $H$ implies the smoothness of $A$. This is the first step towards the proof of Theorem \ref{th seq}. We begin by a known result on the finiteness of the cohomological dimension \cite[Proposition 3.2]{Bic16}.

\begin{prop}\label{prop:finitecd}
Let $k \longrightarrow B \overset{i}{\longrightarrow} A \overset{p}{\longrightarrow} H \longrightarrow k$ be an exact sequence of Hopf algebras with bijective antipodes such that $A$ is flat as left and right $B$-module. We have $\cd(A)\leq \cd(B)+\cd(H)$, and hence if $\cd(B)$ and $\cd(H)$ are finite, so is $\cd(A)$.
\end{prop}

\begin{proof}
Under those assumptions, we can use Theorem \ref{thm : ss ext}. Thus, for every left $A$-module $M$, we get a spectral sequence $$E^{p,q}_2 = \ext_{H}^p({}_\varepsilon k,\ext_{B}^q({}_\varepsilon k,M)) \implies \ext^{p+q}_{A}({}_\varepsilon k,M).$$
If $\cd(B)$ or $\cd(H)$ is infinite, there is nothing to show, so we assume that these are finite.
 For $p > \cd(H)$ or $q> \cd(B)$, we have $E^{pq}_2 = \{0\} $. Now, if we denote $d_2^{p,q} : E^{pq}_2 \to E^{p+2,q-1}_2$ the differential on the second page of the spectral sequence, we get that for $p> \cd(H) $ or $q>\cd(B)$ the maps $d_2^{p,q} : E^{p,q}_2 \to E^{p+2,q-1}_2$ and $d_2^{p-2,q+1} : E^{p-2,q+1}_2 \to E_2^{p,q}$ are both $0$ hence $\{0\} = E_2^{p,q} \simeq E_3^{p,q} \simeq ..\simeq \ext^{p+q}_{A}({}_\varepsilon k,M)$. We obtain that $\ext^{n}_{A}({}_\varepsilon k,M)=\{0\}$
 for $n>\cd(B)+\cd(H)$, and hence $\cd(A) \leq \cd(H)+\cd(B)$.    
\end{proof}

\begin{rem}
    We will prove in the following section that if $B$ is homologically smooth and $H$ has finite cohomological dimension, then we have $\cd(A) = \cd(B) + \cd(H)$. This generalizes \cite[Theorem 5.5]{bieri_homological_1983}.
\end{rem}

\begin{thm}\label{thm : smoothness}
Let $k \longrightarrow B \overset{i}{\longrightarrow} A \overset{p}{\longrightarrow} H \longrightarrow k$ be an exact sequence of Hopf algebras with bijective antipodes such that $A$ is faithfully flat as left and right $B$-module. If $H$ and $B$ are homologically smooth algebras, then $A$ is homologically smooth.
\end{thm}

\begin{proof}
We know from Proposition \ref{prop:finitecd} that $A$ has finite cohomological dimension, hence to prove that $A$ is homologicallly smooth, 
it remains to prove, by Proposition \ref{prop:FP-FPinfty} and Theorem \ref{duality Hopf}, that $k_\varepsilon $ is of type $FP_\infty$ as a right $A$-module. 

To prove that $k_\varepsilon $ is of type $FP_\infty$ as an $A$-module, we will use the characterization of Proposition \ref{characterization FP}, and hence we consider the $A$-module $M = \prod A$ where $\prod$ is an arbitrary direct product. We have to show that for $n\geq 1$, one has $\tor_{n}^A(k_\varepsilon ,M)=0$ and that $\tor_{0}^A(k_\varepsilon,M) \simeq  k_\varepsilon \otimes_A \left(\prod A\right)  \simeq \prod k_\varepsilon \simeq \prod \tor_0^{A}(k_\varepsilon,A)$. 


Under the hypothesis that $A$ is faithfully flat as a right and left $B$-module and regarding Remark \ref{ff=proj}, we can use Theorem \ref{thm : ss tor}. Thus, for every left $A$-module $M$ we get a spectral sequence $$E_{pq}^2 = \tor^H_p(k_\varepsilon,\tor_q^B(k_\varepsilon,M)) \implies \tor_{p+q}^A(k_\varepsilon,M).$$
The $B$-module $k_\varepsilon$ is of type $FP_\infty$ hence $\tor^B_q(k_\varepsilon,-)$ commutes with direct products. Moreover, $A$ is flat as a left $B$-module thus for $q\geq1$, we get that $$\tor^B_q(k_\varepsilon,\prod A) \simeq \prod \tor^B_q(k_\varepsilon,A) = \{0\}.$$
Hence we have $E_{pq}^2=\{0\}$ for $q>0$.

If $\pi_M$ (resp. $\pi_{A}$) denotes the natural projection of $M$ (resp. $A$) onto $\tor_0^B(k_\varepsilon,M) \simeq M/[M,B]$ (resp. $\tor^B_0(k_\varepsilon,A) \simeq A/[A,B]$), the natural isomorphism $\mu : \tor_0^B(k_\varepsilon,M) \to \prod \tor_0^B(k_\varepsilon,A)$ is given by $\mu(\pi_M\bigl(x_i\bigr)_i) = \bigl(\pi_{A}(x_i)\bigr)_i$ for any $x_i \in A$. Hence, for $h = p(a) \in H$, we have $$\mu(h\cdot \pi_M\bigl(x_i \bigr)_i) = \mu(\pi_M\bigl(ax_i \bigr)_i) = \bigl(\pi_{A}(ax_i)\bigr)_i = h\cdot\bigl((\pi_{A}(x_i))_i\bigr).$$ 
The natural isomorphism $\tor_0^B(k_\varepsilon,M) \simeq \prod \tor^B_0(k_\varepsilon,A)$ is thus an isomorphism of $H$-modules. Therefore, since the algebras $B$ and $H$ are homologically smooth, we get, using Proposition \ref{characterization FP}, for $p \geq 0$ 
$$E_{p0}^2=\tor_p^H(k_\varepsilon,\tor_0^B(k_\varepsilon,M)) \simeq \tor_p^H(k_\varepsilon,\prod\tor_0^B(k_\varepsilon,A)) \simeq \prod \tor_p^H(k_\varepsilon,\tor_0^B(k_\varepsilon,A)).$$
The $A$-module $A$ is projective, hence Proposition \ref{prop : acyc homo} ensures that \[\tor_p^H(k_\varepsilon,\tor_0^B(k_\varepsilon,A))=\{0\} \ \text{for $p\geq 1$.}\] Hence for $p\geq 1$ we have
$$E_{p,0}^2=\tor_p^H(k_\varepsilon,\tor_0^B(k_\varepsilon,M)) \simeq \prod \tor_p^H(k_\varepsilon,\tor_0^B(k_\varepsilon,A)) = \{0\}$$ 
We get that $E^2_{p,q}= \tor^H_p(k_\varepsilon,\tor_0^B(k_\varepsilon,M)) = \{0\}$ for $(p,q) \neq (0,0)$, and hence the spectral sequence ensures that $\HH_{n}(A,M)= \{0\}$ for $n>0$.


Finally, using the isomorphism in the proof of Theorem \ref{thm : ss tor} and the fact that $k_\varepsilon$ is of type $FP_\infty$ as a right $H$-module and a right $B$-module, we obtain 
\[
\begin{aligned} \tor_0^A(k_\varepsilon, \prod A) &\simeq \tor_0^H(k_\varepsilon,\tor_0^B(k_\varepsilon,\prod A))\\ &\simeq \tor^H_0(k_\varepsilon,\prod\tor^B_0(k_\varepsilon,A))\\ &\simeq \prod\tor^H_0(k_\varepsilon,\tor^B_0(B,A))\\ &\simeq \prod \tor_0^A(k_\varepsilon,A).
\end{aligned}
\]
It is not difficult to check that the above isomorphism is the natural map in the third item of Proposition \ref{characterization FP}. We thus conclude from Proposition \ref{characterization FP} that $A$ is homologically smooth.
\end{proof}

\section{Homological duality for exact sequences of Hopf algebras}\label{sec : duality ex seq H}

This section is dedicated to the proof of Theorem \ref{th seq}. 

We begin with some generalities.
Let $B$ be an algebra, let $M$ and $N$ be left $B$-modules. We consider the natural morphism $\phi : \Hom_B(N,B)\otimes_B M \to \Hom_B(N,M)$ defined by $\phi (f\otimes m) (n) = f(n)m$ for $m\in M,\ f\in \Hom_B(N,B)$ and $n\in N$. The following is a right version of \cite[Lemma 5.2]{bieri_homological_1983}, we give a proof for the sake of completeness.

\begin{prop}\label{prop : Bi 1}
    Let $B$ be an algebra, let $M$ be a left $B$-module and $N$ a finitely generated projective left $B$-module. The morphism $\phi:  \Hom_B(N,B) \otimes_B M \to \Hom_B(N,M)$ defined above is an isomorphism.
\end{prop}

\begin{proof}
    There exists $B$-modules $L$ and $X$ such that $L$ is free and $L \simeq N \oplus X$. By naturality of $\phi$, we only need to show that $\phi : \Hom_B(L,B) \otimes_B M \to \Hom_B(L,M)$ is bijective and again by naturality, we only need to show that it is an isomorphism for $L\simeq B$. In this case $\phi$ is the identity hence it is bijective. 
\end{proof}

Now,  we replace the left $B$-module $N$ by a $B$-projective resolution $P^\bullet \to N$. This yields natural morphisms $\phi^i : \ext^i_B(N,B) \otimes_B M \to \ext_B^i(N,M)$. Under appropriate assumptions on $B$ and $M$, these morphisms are bijective.

\begin{prop}\label{prop : Bi 2}
    Let $B$ be algebra and $M$ and $N$ left $B$-modules. If $N$ is of type $FP_\infty$ and $M$ is flat, then the morphisms $$ \phi^i : \ext^i_B(N,B) \otimes_B M \to \ext_B^i(N,M)$$ are isomorphisms for any $i\geq 0$.
\end{prop}

\begin{proof}
    Let $P_\bullet \to N$ be a projective resolution with $P_i$ finitely generated for $i\in \mathbb{N}$. Using Proposition \ref{prop : Bi 1}, we have natural isomorphisms $\Hom_B(P_i,B) \otimes_B M \simeq \Hom_B(P_i,M)$ for $i \in \mathbb{N}$.

    This yields to isomorphisms in cohomology $\Homo^i\bigl(\Hom_B(P_\bullet,B) \otimes_B M\bigr) \simeq \Homo^i\bigl(\Hom_B(P_\bullet,M)\bigr) = \ext^i_B(N,M)$ for $i \in \mathbb{N}$. Finally, using flatness of $M$, we have $$\Homo^i\bigl(\Hom_B(P_\bullet,B) \otimes_B M\bigr) \simeq \Homo^i\bigl(\Hom_B(P_\bullet,B)\bigr) \otimes_B M \simeq \ext^i_B(N,B) \otimes_B M$$ which concludes the proof.
\end{proof}

\begin{prop}\label{prop : Bi 4}
    Let $k \longrightarrow B \overset{i}{\longrightarrow} A \overset{p}{\longrightarrow} H \longrightarrow k$ be an exact sequence of Hopf algebras with bijective antipodes, let $X$ be a right $A$-module and let $M$ be a left $A$-module. Then there exists a left $A$-module structure on $X\otimes_B M$ defined by
    \[a \cdot (x \otimes_B m) =xS^{-1}(a_{(2)}) \otimes_B a_{(1)}m, \  a\in A, \ x\in X, m\in M. \]
For $M=A$, endowing $X\otimes H$ with the structure of left $A$-module given by $a' \cdot (x \otimes p(a)) = x \otimes p(a' a)$ for any $a',a\in A$ and $x\in X$,  we have an isomorphism of left $A$-modules $u : X\otimes H \to X\otimes_B A$.    
\end{prop}

\begin{proof}
We first check that the left $A$-action on $X\otimes_B M$ is well defined. For $ a\in A$, $x\in X$, $m \in M$ and $b\in B$, we have 
    \begin{align*}
        a \cdot (xb \otimes_B m) &= xbS^{-1}(a_{(2)}) \otimes_B a_{(1)}m\\
        &= x\varepsilon(a_{(3)})bS^{-1}(a_{(2)}) \otimes_B a_{(1)}m\\
        &= xS^{-1}(a_{(4)})a_{(3)}bS^{-1}(a_{(2)}) \otimes_B a_{(1)}m\\
        &= xS^{-1}(a_{(4)}) \otimes_B a_{(3)}bS^{-1}(a_{(2)})a_{(1)}m \quad {\rm using ~ Lemma \ref{lem : B stable}}\\
        &= xS^{-1}(a_{(2)}) \otimes_B a_{(1)}bm\\
        &= a \cdot(x\otimes_B bm) 
    \end{align*}
 and  hence the $A$-action on $X\otimes_B M$ is well defined.
Let
    \begin{align*}
        u : X\otimes H \ & \longrightarrow X\otimes_B A\\
        x \otimes p(a) & \longmapsto xS^{-1}(a_{(2)}) \otimes_B a_{(1)}.  
    \end{align*}
    It is not difficult to check that $u$ is well-defined.
    For $a,a'\in A$ and $x\in X$, we have $$u(a'\cdot(x\otimes p(a))) = u(x\otimes p(a'a)) = xS^{-1}(a_{(2)})S^{-1}(a'_{(2)}) \otimes_B a'_{(1)}a_{(1)} = a' \cdot u(x\otimes p(a))$$ and hence $u$ is $A$-linear. Now  consider the linear map 
    \begin{align*}
        v : X\otimes A \ & \longrightarrow X\otimes H\\
        x \otimes a & \longmapsto xa_{(2)} \otimes p(a_{(1)})  
    \end{align*}
    For $x\in X, \ a\in A$ and $b\in B$, we have $$v(xb\otimes a) = xb a_{(2)} \otimes p(a_{(1)}) = xb_{(2)}a_{(2)} \otimes p(b_{(1)})p(a_{(1)}) = v(x\otimes ba)$$ since $B = {}^{co H} A$. Thus $v$ induces a map
     \begin{align*}
        v : X\otimes_B A \ & \longrightarrow X\otimes H\\
        x \otimes_B a & \longmapsto xa_{(2)} \otimes p(a_{(1)})  
    \end{align*}
    For $x\in X$ and $a\in A$, we have 
    \begin{align*}
u\circ v (x\otimes_B a) &= u\bigl(xa_{(2)} \otimes p(a_{(1)})\bigr) = xa_{(3)} S^{-1}(a_{(2)}) \otimes_B a_{(1)} = x\otimes a, \ \text{and} \\
v\circ u (x\otimes p(a)) &= v\bigl(xS^{-1}(a_{(2)}) \otimes_B a_{(1)}\bigr) = xS^{-1}(a_{(3)}) a_{(2)} \otimes p(a_{(1)}) = x\otimes p(a).
    \end{align*}
    Hence $u$ is an isomorphism of $A$-modules.
\end{proof}

Let $k \longrightarrow B \overset{i}{\longrightarrow} A \overset{p}{\longrightarrow} H \longrightarrow k$ be an exact sequence of Hopf algebras with $A$ projective as a left $B$-module. If $M$ is a left $A$-module, recall from the previous section that $\ext_B^\bullet({}_\varepsilon k ,M)$ is endowed with a natural left $A$-module structure defined at the cochain level by $$\bigl(a \rightharpoonup f \bigr) (x) = a_{(1)}f(S(a_{(2)})x)$$ for $a\in A, \ f\in \Hom_B(P_q,M)$ and $x\in P_q$ where $P_\bullet$ is an $A$-projective resolution of ${}_\varepsilon k$. This is this action of $A$ induced by the action of $H$ defined in Proposition \ref{structure 2}.

Moreover $\ext_B^\bullet({}_\varepsilon k,B)$ is endowed with a natural structure of right $A$-module, extending its natural $B$-module structure, defined as follows. Consider an $A$-projective resolution $P_\bullet$ of ${}_\varepsilon k$. For any $q \geq 0$ and for $a\in A,\ f\in \Hom_B(P_q,B), \ x \in P_q$, put $$(f \leftharpoonup  a)(x) = S(a_{(2)})f(S^2(a_{(1)})x)a_{(3)}. $$ 
By Lemma \ref{lem : B stable} we have $S(a_{(2)})f(S^2(a_{(1)})x)a_{(3)} \in B$ for $a\in A,\ f\in \Hom_B(P_q,B)$ and $ x \in P_q$. Moreover, one easily checks using again Lemma \ref{lem : B stable} that $f\leftharpoonup a$ is a morphism of $B$-modules. Hence the above formula defines an $A$-module structure on the complex $\Hom_B(P^\bullet,B)$, and hence on $\ext_B^\bullet({}_\varepsilon k,B)$.

\begin{prop}\label{prop : Bi 3}
    Let $k \longrightarrow B \overset{i}{\longrightarrow} A \overset{p}{\longrightarrow} H \longrightarrow k$ be an exact sequence of Hopf algebras such that $A$ is faithfully flat as a left $B$-module, and let $M$ be a left $A$-module. The morphisms $$ \phi^i :  \ext^i_B({}_\varepsilon k,B) \otimes_B M \longrightarrow \ext_B^i({}_\varepsilon k ,M)$$ defined above are morphisms of left $A$-modules, where the left $A$-module structure on $\ext_B^i({}_\varepsilon k,B) \otimes_B M$ is the one of Proposition \ref{prop : Bi 4}, i.e. is given by $$a\rightarrow ( [f] \otimes_B m)  = [f] \leftharpoonup S^{-1}(a_{(2)}) \otimes_B a_{(1)}m$$ for $a\in A,\ m\in M$ and $[f]\in \ext_B^i({}_\varepsilon k,B)$.
\end{prop}

\begin{proof}
    Let $P_\bullet$ be a projective resolution of ${}_\varepsilon k$ by $A$-modules. At the cochain level, we have for $a\in A,\ m\in M,\ f\in \Hom_B(P_i,B)$ and $x\in P_i$
    \begin{align*}
        \phi^i\biggl( a\rightarrow ( f \otimes m)\biggr) (x) &= \phi^i\biggl( f \leftharpoonup S^{-1}(a_{(2)}) \otimes a_{(1)}m\biggr) (x) \\ 
        &= ( f \leftharpoonup S^{-1}(a_{(2)})) (x) a_{(1)}m\\
        &=  a_{(3)} f(S(a_{(4)})x) S^{-1}(a_{(2)}) a_{(1)} m\\
        &= a_{(1)}f(S(a_{(2)})x) m \\
        &= \biggl( a\rightharpoonup \bigl(\phi^i (f\otimes m) \bigr)   \biggr) (x) 
    \end{align*}
 and from this we deduce the announced result.   
\end{proof}

We can now prove the homological equality announced in Section \ref{Sec:Smoothness}.

\begin{thm}\label{coro : cd equality}
    Let $k \longrightarrow B \overset{i}{\longrightarrow} A \overset{p}{\longrightarrow} H \longrightarrow k$ be an exact sequence of Hopf algebras with bijective antipodes such that $A$ is faithfully flat as a left and right $B$-module. If $B$ is homologically smooth and $H$ has finite cohomological dimension, then we have $$\cd(A) = \cd(B) + \cd(H).$$ 
\end{thm}

\begin{proof}
    Let $n = \cd(B)$ and $m = \cd(H)$. There is a free left $H$-module $F = V \otimes H$ such that $\ext^m_H({}_\varepsilon k, F) \neq 0$. Following Theorem \ref{thm : ss ext} for the left $A$-module $V\otimes A$, we get the spectral sequence $$E^{p,q}_2 = \ext_{H}^p({}_\varepsilon k,\ext_{B}^q({}_\varepsilon k,V\otimes A)) \implies \ext^{p+q}_{A}({}_\varepsilon k,V\otimes A).$$ 
    For $p>\cd(H)$ or $q>\cd(B)$, we have $E_2^{p,q} = 0$ hence, reasoning as in the proof of Proposition \ref{prop:finitecd}, we get that $$\ext^{m+n}_A({}_\varepsilon k, V \otimes A)\simeq E^{p,q}_2 \simeq \ext^m_H({}_\varepsilon k, \ext^n_B({}_\varepsilon k, V\otimes A)).$$

    Then, using Proposition \ref{prop : Bi 2} and Proposition \ref{prop : Bi 3} we have $\ext^n_B({}_\varepsilon k, V\otimes A) \simeq \ext^n_B({}_\varepsilon k, B) \otimes_B (V\otimes A)$ as left $A$-modules. Since the action on $V\otimes A$ is simply multiplication on $A$, we have $\ext^n_B({}_\varepsilon k, B) \otimes_B (V\otimes A) \simeq (\ext^n_B({}_\varepsilon k, B) \otimes V) \otimes_B A$ and using Proposition \ref{prop : Bi 4} with $X = \ext^n_B({}_\varepsilon k, B) \otimes V$, we get that $\ext^n_B({}_\varepsilon k, V\otimes A) \simeq \ext^n_B({}_\varepsilon k, B) \otimes V \otimes H$ as left $A$-modules where the $A$-action on $\ext^n_B({}_\varepsilon k, B) \otimes V \otimes H$ is induced by the multiplication in $H$ and the surjective morphism $p:A \to H$. Therefore, we have $\ext^n_B({}_\varepsilon k, V\otimes A) \simeq \ext^n_B({}_\varepsilon k, B) \otimes V \otimes H$  as $H$-modules where the action on the left term is the one described in Proposition \ref{structure 2} and the action on the right term is the multiplication in $H$. Hence $\ext^{m+n}_A({}_\varepsilon k, V \otimes A) \simeq \ext^m_H({}_\varepsilon k, \ext^n_B({}_\varepsilon k, B) \otimes V \otimes H)$.

    Finally, since $\cd(B) = n$ and $B$ is homologically smooth, we have $\ext^n_B({}_\varepsilon k, B)\neq 0$, thus $\ext^n_B({}_\varepsilon k, B) \otimes V \otimes H$ can be written as a direct sum of copies of $V\otimes H$ and $\ext^m_H({}_\varepsilon k, \ext^n_B({}_\varepsilon k, B) \otimes V \otimes H)$ is non zero which concludes the proof.
\end{proof}

\begin{prop}\label{iso lemma hopf}
    Let $k \longrightarrow B \overset{i}{\longrightarrow} A \overset{p}{\longrightarrow} H \longrightarrow k$ be an exact sequence of Hopf algebras with bijective antipodes such that $A$ is faithfully flat as a left and right $B$-module and let $M$ be a left $A$-module. For $q\geq 0$, there is an isomorphism of $H$-modules $$\ext^q_{B}({}_\varepsilon k,B) \otimes H \simeq \ext_{B}^q({}_\varepsilon k,A)$$ where the $H$-module structure on $\ext^q_{B}({}_\varepsilon k,B) \otimes H$ is induced by left multiplication on $H$ and the $H$-module structure on $\ext_{B}^q({}_\varepsilon k ,A)$ is the one defined in Proposition \ref{structure 2}.     
    This isomorphism is given by $$[f]\otimes p(a) \mapsto p(a) \cdot[f]$$ for $[f]\in \ext^q_{B}({}_\varepsilon k,B)$ and $a\in A$.
\end{prop}

\begin{proof}
    Using Proposition \ref{prop : Bi 4} with $X = \ext_B^q({}_\varepsilon k,B)$ and Proposition \ref{prop : Bi 2} for the flat left $B$-module $M = A$, we get the isomorphism $\ext^q_{B}(k_\varepsilon,B) \otimes H \simeq \ext_{B}^q(k_\varepsilon,A)$. 

    Let $P_\bullet$ be a projective resolution of ${}_\varepsilon k$ by $A$-modules, at the cochain level, we have for $a\in A,\ f\in \Hom_B(P_i,B)$ and $x\in P_i$   
    \begin{align*}
        \phi^i \circ u \biggl(( f \otimes p(a))\biggr) (x) &= \phi^i\biggl( f \leftharpoonup S^{-1}(a_{(2)}) \otimes a_{(1)}\biggr) (x) \\ 
        &= ( f \leftharpoonup S^{-1}(a_{(2)})) (x) a_{(1)}\\
        &= a_{(1)}f(S(a_{(2)})x) m \\
        &= (p(a) \cdot f) (x) 
    \end{align*}
    Hence, this isomorphism is given by the above formula and it is clearly $H$-linear which concludes the proof.
\end{proof}

Using this isomorphism of $H$-modules, we get the following isomorphisms between cohomology spaces.

\begin{prop}\label{iso seq}
   Let $k \longrightarrow B \overset{i}{\longrightarrow} A \overset{p}{\longrightarrow} H \longrightarrow k$ be an exact sequence of Hopf algebras with bijective antipodes such that $A$ is faithfully flat as a left and right $B$-module. Assume that $H$ and $B$ are smooth Hopf algebras. For $p,q\geq 0$, we have an isomorphism 
    $$E_2^{p,q} = \ext_{H}^p ({}_\varepsilon k, \ext_{B}^q({}_\varepsilon k,A)) \simeq \ext_{B}^q({}_\varepsilon k,B) \otimes \ext_{H}^p({}_\varepsilon k,H).$$
\end{prop}

\begin{proof}
    Using the isomorphism of $H$-modules in Proposition \ref{iso lemma hopf}, we have that for $p,q\geq 0$ $$E_2^{p,q}=\ext_{H}^p ({}_\varepsilon k, \ext_{B}^q({}_\varepsilon k,A)) \simeq \ext_{H}^p\left({}_\varepsilon k,\ext_{B}^q({}_\varepsilon k,B) \otimes H\right).$$
    Hence, using the smoothness of $H$, Proposition \ref{FP ext} and the fact that the $H$-module structure on $\ext_{B}^q({}_\varepsilon k,B) \otimes H$ is given by multiplication on $H$, we get the result.
\end{proof}

This allows to prove Theorem \ref{th seq}

\begin{proof}[Proof of Theorem \ref{th seq}]
     Let $k \longrightarrow B \overset{i}{\longrightarrow} A \overset{p}{\longrightarrow} H \longrightarrow k$ be an exact sequence of Hopf algebras with bijective antipodes such that $A$ is faithfully flat as a left and right $B$-module, and with $H$ and $B$ homologically smooth.
     
   Let us first assume that $B$ and $H$ have homological duality. 
    Using Theorem \ref{thm : smoothness}, we get that $A$ is homologically smooth, and (see Theorem \ref{duality Hopf}) we have to prove that $\ext_A^{n}({}_\varepsilon k,A) = 0$ if $n\neq \cd(A)$. 

    Since $H$ and $B$ have homological duality,  the isomorphism of Proposition \ref{iso seq} gives that $E_2^{p,q} =0$ if $p\neq \cd(H)$ or $q\neq \cd(B)$. Hence, $E_2^{\cd(H),\cd(B)}$ is the only non-zero space on the second page of the spectral sequence and we get that $\ext_A^{n}({}_\varepsilon k,A) = 0$ if $n \neq {\cd(B)+\cd(H)}$, while $\cd(A)=\cd(B)+\cd(H)$ by Theorem \ref{coro : cd equality}: we conclude that $A$ has homological duality.


    Conversely, assume that $A$ has homological duality. Let $r$ and $s$ be the smallest integers such that $\ext^r_{H}({}_\varepsilon k, H) \neq 0$ and $\ext_B^s({}_\varepsilon k,B) \neq 0$.
    Using the isomorphism of Proposition \ref{iso seq} and a corner argument, we get that $$E_2^{r,s} \simeq \ext_A^{r+s}({}_\varepsilon k,A) \simeq \ext_B^s({}_\varepsilon k,B) \otimes  \ext_{H}^r({}_\varepsilon k,H) \neq 0.$$ Hence, we have $\cd(A) = r+s$. We have $r\leq \cd(H)$ and $s\leq \cd(B)$, and $\cd(A) = \cd(B) + \cd(H)$ by Theorem \ref{coro : cd equality}, hence $r=\cd(H)$ and $s=\cd(B)$. We conclude from Theorem \ref{duality Hopf} that $H$ and $B$ have homological duality.
\end{proof}

We conclude this section with the case of twisted Calabi-Yau Hopf algebras.

\begin{coro}\label{coro:tCY}
    Let $k \longrightarrow B \overset{i}{\longrightarrow} A \overset{p}{\longrightarrow} H \longrightarrow k$ be an exact sequence of Hopf algebras with bijective antipodes and such that $A$ is faithfully flat as a left and right $B$-module. Assume that $H$ and $B$ are twisted Calabi-Yau. Then $A$ is twisted Calabi-Yau  in dimension $\cd(B) + \cd(H)$.
\end{coro}

\begin{proof}
    We already know Theorem \ref{th seq} that $A$ has homological duality in dimension $\cd(B) + \cd(H)$. Moreover we know fom the proof of Theorem \ref{th seq} that  \[\ext^{\cd(A)}_A ({}_\varepsilon k,A) \simeq \ext^{\cd(H)}_H ({}_\varepsilon k,H) \otimes \ext^{\cd(B)}_B ({}_\varepsilon k,B).\] By Proposition \ref{Prop : CY Hopf}, $\ext^{\cd(H)}_H ({}_\varepsilon k,H)$ and $\ext^{\cd(B)}_B ({}_\varepsilon k,B)$ are one-dimensional, hence $\ext^{\cd(A)}_A ({}_\varepsilon k,A)$ is one-dimensional and this concludes the proof, again using Proposition \ref{Prop : CY Hopf}.
\end{proof}

\section{Examples}

\subsection{The universal pointed Hopf algebra coacting on $k[x,y]$} 

First consider the algebra $H_1$ presented by generators $x,g, g^{-1}$ and relations $gg^{-1} = g^{-1}g = 1$, endowed with the Hopf algebra structure 
$$\Delta(x) = 1\otimes x + x\otimes g,\quad \Delta(g) = g\otimes g,\quad \varepsilon(x) = 0,\quad \varepsilon(g) = 1,\quad S(g) = g^{-1},\quad S(x) = -xg^{-1} .$$ 


\begin{prop}\label{duality H_1}
    The Hopf algebra $H_1$  has homological duality in dimension $1$.
\end{prop}

\begin{proof}
   We have  $H_1^+ = (g-1) H_1 \oplus x H_1$, hence $H_1^+$ is a finitely generated free $H_1$-module and $0 \to H_1 ^+ \to H_1  \to k \to 0$ 
    is a finitely generated free resolution of $_\varepsilon k$. Hence by Theorem \ref{smooth Hopf} the Hopf algebra $H_1$ is smooth, and since
    moreover it is clear that $\ext^0({}_\varepsilon k,H_1) \simeq \Hom_{H_1}({}_\varepsilon k, H_1) = 0$, Theorem \ref{duality Hopf} gives the result.
\end{proof}

We now define $H_2$ to be the algebra presented by generators $u_{11}, u_{12}, u_{22}$ with $u_{11}$ central in $H_2$ and $u_{11}$ and $u_{22}$ invertible and we endow it with the Hopf algebra structure 
$$\Delta(u_{11}) = u_{11}\otimes u_{11}, \quad \Delta(u_{12}) = u_{11}\otimes u_{12} + u_{12}\otimes u_{22}, \quad \Delta(u_{22}) = u_{22} \otimes u_{22}, \quad \varepsilon(u_{ij}) = \delta_{ij}.$$ 
The Hopf algebra $H_2$ is pointed since it is generated by group-like and skew-primitive elements.

There has been recently some interest on coactions of pointed algebras on commutative algebras, see \cite{krahmer_nodal_2017,brown_plane_2022}.
In this setting, the motivation for introducing the above Hopf algebra $H_2$ is that it is in some sense a universal pointed Hopf algebra coacting on $k[x,y]$ in a grading preserving way, as follows.

\begin{prop}
There is a grading preserving $H_2$-comodule algebra structure  on $k[x,y]$ defined by
\begin{align*}
    \rho : k[x,y] \ & \longrightarrow k[x,y] \otimes H_2 \\
x , \ y & \longmapsto x\otimes u_{11}, \ x\otimes u_{12} + y\otimes u_{22}  
\end{align*}
If $A$ is a pointed Hopf algebra coacting on $k[x,y]$ in a grading preserving way,
 then there exist  an $H_2$-comodule algebra structure given by a  coaction $\rho' : k[x,y] \to k[x,y] \otimes H_2$, and a Hopf algebra map $f :  H_2 \to A$ making the following diagram commute
$$\xymatrix{k[x,y] \ar[r]^-{\rho'} \ar[dr]_{\theta} & k[x,y] \otimes  H_2 \ar[d]^{{\rm id}\otimes f}\\
& k[x,y] \otimes A & 
}$$
where $\theta$ denotes the given coaction of  $A$.
\end{prop}

\begin{proof}
The existence of the announced algebra map $\rho$ follows from the centrality of $u_{11}$, and it is clear that $\rho$ defines a grading preserving right $H_2$-comodule structure on $k[x,y]$.   

Now assume given a pointed Hopf algebra $A$  coacting on $k[x,y]$ in a grading preserving way. Then $V=kx\oplus ky$ is a subcomodule, and since $A$ is pointed, has a one-dimensional subcomodule.
Hence there exists a basis $\{x',y'\}$ of $V$ (with $k[x,y]=k[x',y']$) such that $\theta(x')=x'\otimes g$ and $\theta(y')= x'\otimes z + y'\otimes h$ for some $g,h,z\in A$. Since $\theta$ is an algebra map and $\{x'^2, x'y'\}$ is a linearly independent set, we see that $g$ commutes with $z$ and $h$. Defining the coaction $\rho'$ in the same way as $\rho$ before (replacing $x,y$ by $x',y'$), it is then a direct verification that there exists a unique Hopf algebra map $f : H_2 \to A$ making the above diagram commute.
\end{proof}

We now prove that $H_2$ has homological duality in dimension $2$ using Theorem \ref{th seq}.

\begin{prop}
    The Hopf algebra $H_2$ has homological duality in dimension $2$. 
\end{prop}

\begin{proof}
    Consider the sequence $k \to k\mathbb{Z} \overset{i}{\to} H_2 \overset{p}{\to} H_1 \longrightarrow k$
    where $i$ and $p$ are the Hopf algebra maps defined by 
    $i : k\mathbb{Z} \to H_2$, $t\mapsto u_{11}$ and 
    \begin{align*}
    p :  H_2 &\longrightarrow H_1\\
 u_{11}, u_{12}, u_{22}  & \longmapsto 1, x, g
\end{align*}
Using the diamond lemma we easily see that $i$ is injective, and the centrality of $u_{11}$ ensures that $p$ is the quotient map of $H_2$ by $i(k\mathbb Z)^+H_2=H_2i(k\mathbb Z)^+$. Moreover, using again the diamond lemma, we see
that $H_2$ is free as a left or right $i(k\mathbb Z)$-module, hence faithfully flat and by \cite[Proposition 3.4.3]{montgomery_hopf_1993}, we conclude that 
    this is an exact sequence of Hopf algebras. 
    It is well known that $k\mathbb{Z}$ has homological duality in dimension $1$ and Proposition \ref{duality H_1} assure that $H_1$ has homological duality in dimension $1$.  We conclude using Theorem \ref{th seq}.
\end{proof}

\subsection{Quantum groups of $GL(2)$ representation type}

In this subsection, we use Corollary \ref{coro:tCY} to show that the Hopf algebras $\mathcal{G}(E,E^{-1})$ introduced in \cite{mrozinski_quantum_2014} are twisted Calabi-Yau in dimension $4$. This was known from \cite{wang_calabi-yau_2023}, but the argument given here, based on a appropriate exact sequence, is shorter.

Let $n\geq 2$ and let $A,B\in GL_n(k)$. Consider the universal algebra $\mathcal{G}(A,B)$ with generators $(x_{ij})_{1\leq i,j\leq n}$, $d, d^{-1}$ satisfying the relations $$x^t A x = Ad, \quad xBx^t = Bd,\quad dd^{-1}=1 = d^{-1}d $$ 
where $x$ is the matrix $(x_{ij})_{1\leq i,j \leq n}$.
It can be endowed with the following  Hopf algebra structure $$\Delta(x_{ij}) = \sum\limits_{k=1}^n x_{ik}\otimes x_{kj}, \quad \Delta(d^{\pm}) = d^{\pm}\otimes d^{\pm},$$
$$\varepsilon(x_{ij}) = \delta_{ij}, \quad \varepsilon(d^{\pm}) = 1,\quad S(x) = d^{-1}A^{-1}x^t A,\quad S(d^{\pm}) = d^{\mp}. $$
For $q\in k^*$ and $A_q = \begin{pmatrix} 0 &1 \\ -q &0\end{pmatrix} \in GL_2(k)$, we have $\mathcal{G}(A_q,A_q) = \mathcal{O}(GL_q(2))$.

\begin{prop}
    Let $n\geq 2$ and $E \in GL_n(k)$. The Hopf algebra $\mathcal{G}(E,E^{-1})$ is a twisted Calabi-Yau algebra of dimension $4$.
\end{prop}

\begin{proof}
    Let $\mathcal{B}(E)$ (\cite{dubois-violette_quantum_1990}) be the Hopf algebra presented by generators $(u_{ij})_{1\leq i,j\leq n}$ and relations $$E^{-1}u^t E u= I_n =  u E^{-1} u^t E.$$ 
    We consider the sequence of Hopf algebra maps $$k \longrightarrow k\mathbb{Z} \overset{\iota}{\longrightarrow} \mathcal{G}(E,E^{-1}) \overset{\pi}{\longrightarrow} \mathcal{B}(E) \longrightarrow k$$ 
    where $\iota : k\mathbb{Z} \to \mathcal{G}(E,E^{-1})$ and $\pi : \mathcal{G}(E,E^{-1}) \to \mathcal{B}(E)$ are defined by $$\iota(t) = d,\quad \pi(x_{ij}) = u_{ij}\quad {\rm and } \quad\pi(d^\pm) = 1,$$ where $t$ is the generator of $\mathbb{Z}$. We easily see that $d$ is central in $\mathcal{G}(E,E^{-1})$, hence $\iota(k\mathbb{Z})^+ \mathcal{G}(E,E^{-1}) = \mathcal{G}(E,E^{-1})\iota(k\mathbb{Z})^+$ and that $\pi$ is the quotient map of $\mathcal{G}(E,E^{-1})$ by $\mathcal{G}(E,E^{-1})\iota(k\mathbb{Z})^+$. 
    

Reasoning as in \cite[Proposition 11.6]{mrozinski:tel-00948512}, we can get a linear basis of $\mathcal{G}(E,E^{-1})$ using the diamond lemma \cite{bergman78}. We order the set $\{1,\dots, \ n\} \times \{1,\dots, \ n\}$ lexicographically and we set $$Ind(E) = \left\{ I=\bigl((i_1,j_1), \dots,\ (i_m,j_m)\bigr), (i_t,i_{t+1}) \neq (n,u) \ {\rm and} \ (j_t,j_{t+1}) \neq (n,v)\right\}$$ where $(n,u)$ and $(n,v)$ are the maximal elements of $\{1,\dots, \ n\} \times \{1,\dots, \ n\}$ such that $E_{nu}\neq 0$ and $(E^{-1})_{nv}\neq 0$.
Using a suitable presentation of $\mathcal{G}(E,E^{-1})$ and the diamond lemma, we get that the family $\left(d^k x_I \right)_{k\in \mathbb{Z}, I \in Ind(E)}$ is a linear basis of $\mathcal{G}(E,E^{-1})$. In particular $\iota$ is injective and we get that $\mathcal{G}(E,E^{-1})$ is free hence faithfully flat as a left or right module over $\iota(k\mathbb{Z})$. By Remark \ref{rem:exact} we conclude that our sequence is an exact sequence of Hopf algebras.

It is known \cite{bichon_hochschild_2013} that $\mathcal{B}(E)$ is twisted Calabi-Yau in dimension $3$,  hence, using Corollary \ref{coro:tCY}, we get the result.
\end{proof}

\begin{rem}
    In fact, as mentioned before it was recently shown in \cite{wang_calabi-yau_2023}, by computing an explicit resolution of $k$ by free $\mathcal{G}(A,B)$ Yetter-Drinfeld modules that, for $B^tA^tBA\in k^*$, the algebras $\mathcal{G}(A,B)$ are twisted Calabi-Yau in dimension $4$.
\end{rem}

\bibliography{exact_sequences}

\end{document}